%% file: super_hopf.tex
\documentclass[reqno]{gtpart}

\usepackage[utf8]{inputenc}

\usepackage{
	amsthm,
	amsmath,
	amsfonts,
	cancel, 
    float,
    mathrsfs,
    soul, 
}

\usepackage[bottom=3.5cm,left=2.8cm,right=2cm,top=3cm]{geometry}

\input{pics.tex}

\title{On the Hopf algebra of noncommutative\\symmetric functions in superspace}
\author[D. Arcis, C. Gonz\'alez and S. M\'arquez]{Diego Arcis\\Camilo Gonz\'alez\\Sebasti\'an M\'arquez}

\address{Departamento de Matem\'aticas, Universidad Aut\'onoma de Chile -- Sede Talca, 5 Poniente 1670 -- 3460000 Talca, Chile.\textcolor{white}{$\underbrace{1}$}\newline
Instituto de Matem\'aticas, Universidad de Talca -- Campus Norte, Camino Lircay S/N -- 3460000 Talca, Chile.\textcolor{white}{$\underbrace{.}$}\newline
Departamento de Matem\'aticas, Universidad Aut\'onoma de Chile -- Sede Talca, 4 Norte 95 -- 3460000 Talca, Chile.}

\email{diego.arcis@uautonoma.cl}
\email{cgonzalez@inst-mat.utalca.cl}
\email{sebastian.marquez@uautonoma.cl}

\subject{primary}{msc2020}{05E05} 
\subject{primary}{msc2020}{16T05} 
\subject{primary}{msc2020}{16T30} 

\numberwithin{equation}{section}

\begin{document}

\newtheorem{thm}{Theorem}[section]
\newtheorem{crl}[thm]{Corollary}
\newtheorem{dfn}[thm]{Definition}
\newtheorem{lem}[thm]{Lemma}
\newtheorem{pro}[thm]{Proposition}

\theoremstyle{definition}
\newtheorem{exm}[thm]{Example}
\newtheorem{ntn}[thm]{Notation}
\newtheorem{rem}[thm]{Remark}

\def\blue{\color{blue}}\def\red{\color{red}}\def\green{\color{green}}
\newcommand{\stred}[1]{\textcolor{red}{\st{#1}}}

\newcommand{\sNSym}{\mathrm{\bf sNSym}}
\newcommand{\NSym}{\mathrm{\bf NSym}}
\newcommand{\sSym}{\mathrm{\bf sSym}}
\newcommand{\sQSym}{\mathrm{\bf sQSym}}
\newcommand{\QSym}{\mathrm{\bf QSym}}
\newcommand{\SSym}{\mathrm{\bf Sym}}

\newcommand{\mH}{\mathcal{H}}
\newcommand{\mP}{\mathcal{P}}

\newcommand{\Prim}{\mathrm{Prim}}

\newcommand{\tC}{\tilde{C}}
\newcommand{\tH}{\tilde{H}}
\newcommand{\tS}{\tilde{S}}
\newcommand{\tP}{\tilde{P}}

\newcommand{\te}{\tilde{e}}
\newcommand{\tp}{\tilde{p}}

\renewcommand{\th}{\tilde{h}}

\newcommand{\field}{\Q}

\newcommand{\N}{\mathbb{N}}

\newcommand{\adm}{\operatorname{adm}}
\newcommand{\sgn}{\operatorname{sgn}}

\newcommand{\cv}{\operatorname{cv}}
\newcommand{\df}{\operatorname{df}}
\newcommand{\id}{\operatorname{id}}
\newcommand{\rg}{\operatorname{rg}}

\newcommand{\Sym}{\mathfrak{S}}

\begin{abstract}
We study in detail the Hopf algebra of noncommutative symmetric functions in superspace $\sNSym$, introduced by Fishel, Lapointe and Pinto. We introduce a family of primitive elements of $\sNSym$ and extend the noncommutative elementary and power sum functions to superspace. Then, we give formulas relating these families of functions. Also, we introduce noncommutative Ribbon Schur functions in superspace and provide a explicit formula for their product. We show that the dual basis of these function is given by a family of the so--called fundamental quasisymmetric functions in superspace. This allows us to obtain a explicit formula for the coproduct of fundamental quasisymmetric functions in superspace. Additionally, by projecting the noncommutative Ribbon Schur functions in superspace, we define a new basis for the algebra of symmetric functions in superspace. On the other hand, we also show that $\sNSym$ can be realised as a Hopf algebra of trees.
\end{abstract}

\maketitle


\section{Introduction}

The classical ring of symmetric functions $\SSym$ is a very important object in algebraic combinatorics. This has been widely studied due to its richness properties and multiple applications in several areas such as representations theory, algebraic geometry and Lie algebras \cite{Mac99,St99}. Symmetric functions are important not only in mathematics, but also in connection with integrable models in physics \cite{MiJiDa00}. Due to this connection, and motivated by a supersymmetric generalization of the Calogero--Moser--Sutherland model \cite{DeLaMa01,DeLaMaProc04}, it was developed in \cite{DeLaMa03,DeLaMa04} a new class of functions called \emph{symmetric functions in superspace}, which generalize the classic symmetric functions.

A \emph{symmetric function in superspace} is obtained by considering the classic infinite collection of commutative variables $x=(x_1,x_2,\ldots)$ together with an infinite collection of anticommutative variables $\theta=(\theta_1,\theta_2,\ldots)$, that is, $\theta_i\theta_j=-\theta_j\theta_i$. In particular, $\theta_i^2=0$. As it was shown in \cite{DeLaMa06}, the classic bases of $\SSym$, as the monomials, power--sum, elementary and homogeneous symmetric functions, can be naturally extended to superspace. More complex bases as Schur, Jack and Macdonald polynomial were also extended and studied in \cite{BlDeLaMa12,BlMa15,DeLaMa07,JoLa17,DeLaMa11,GoLa20,GaJoLa19}. The combinatorial tools used to study these objects are the so--called \emph{superpartitions}, which generalize the classic partitions of numbers.

A \emph{superpartition} is a pair $\Lambda=(\Lambda^a;\Lambda^s)$, where $\Lambda^a$ is a strictly decreasing partition, including possible a zero, and $\Lambda^s$ is a usual partition. The \emph{degree} $|\Lambda|$ of $\Lambda$ is the sum of the components of $\Lambda^a$ with the ones of $\Lambda^s$. The \emph{fermionic degree} $\df(\Lambda)$ of $\Lambda$ is the length of $\Lambda^a$. In particular, every usual partition can be regarded as a superpartition with null fermionic degree.

Thus, given a superpartition $\Lambda=(\Lambda_1,\ldots,\Lambda_m;\Lambda_{m+1},\ldots,\Lambda_n)$, some of the classical symmetric functions in superspace are defined as follows:\begin{itemize}
\item\emph{Monomial symmetric functions in superspace}:\[m_{\Lambda}(x,\theta)=\sum_{\sigma\in\Sym_n} \theta_{\sigma(1)}\cdots\theta_{\sigma(m)}x_{\theta(1)}^{\Lambda_1}\cdots x_{\sigma(n)}^{\Lambda_n},\]where the sum consider no repeated terms.
\item\emph{Power--sum symmetric functions in superspace}: $p_{\Lambda}=\tp_{\Lambda_1}\cdots\tp_{\Lambda_m}p_{\Lambda_{m+1}}\cdots p_{\Lambda_n}$, where\[\tp_k=\sum_{i=1}^n\theta_ix_i^k\quad\text{and}\quad p_r=\sum_{i=1}^nx_i^r,\quad\text{for}\quad k\geq0, r\geq1.\]
\item\emph{Elementary symmetric functions in superspace}: $e_{\Lambda}=\te_{\Lambda_1}\cdots\te_{\Lambda_m}e_{\Lambda_{m+1}}\cdots e_{\Lambda_n}$, where\[\te_k=m_{(0;1^k)}\quad\text{and}\quad e_r=m_{(\emptyset;1^r)},\quad\text{for}\quad k\geq0, r\geq1.\] 
\item\emph{Complete homogeneous symmetric functions in superspace}: $h_{\Lambda}=\th_{\Lambda_1}\ldots\th_{\Lambda_m}h_{\Lambda_{m+1}}\cdots h_{\Lambda_n}$, where\[\th_k=\sum_{|\Lambda|=k,\,\df(\Lambda)=1}(\Lambda_1+1)m_{\Lambda}\quad\text{ and}\quad h_r=\sum_{|\Lambda|=r,\,\df(\Lambda)=0}m_\Lambda,\quad\text{for}\quad k\geq0, r\geq1.\]
\end{itemize}
Note that if $\Lambda^a=\emptyset$, we obtain a classic symmetric function.

It is well known that $\SSym$ has a Hopf algebra structure \cite{Eh96}. In \cite{FiLaPi19}, the set of symmetric functions in superspace was also given with a Hopf algebra structure $\sSym$, which has explicit formulas for the product, coproduct and antipode in terms of the bases given above. This Hopf algebra is commutative, cocommutative (up a sign), self dual \cite[Proposition 4.8]{FiLaPi19} and contains $\SSym$ as a Hopf subalgebra.

Another important ring related to symmetric functions is the one of \emph{quasisymmetric functions} $\QSym$ \cite{Ges84}, whose Hopf structure has several applications in the theory of symmetric functions such as an expansion of the Macdonald polynomial in terms of the so--called fundamental quasisymmetric functions \cite{Ges84}. The fundamental quasisymmetric functions provide an important basis of $\QSym$, since they share similar properties with Schur functions. The product, coproduct and antipode of these type of functions in superspace is an open problem. Here, we partially lead with this problem. 

The dual structure of $\QSym$, defined in \cite{GKLLRT95}, is the Hopf algebra of \emph{noncommutative symmetric functions} $\NSym$. It was shown in \cite{GKLLRT95} that there is a connection of $\NSym$ with the Solomon's descent algebra, which is applied to the study of formal power series with coefficients in a noncommutative algebra. The usual bases of $\QSym$ and $\NSym$ are indexed by compositions of numbers.

It was shown in \cite{FiLaPi19} that both $\QSym$ and its dual $\NSym$ can be extended to superspace as the Hopf algebras of \emph{quasisymmetric function in superspace} $\sQSym$ and of \emph{noncommutative symmetric functions in superspace} $\sNSym$, respectively. In superspace, the classical bases are indexed by the so--called \emph{dotted compositions}, that is, tuples of nonnegative integers in which some of their components are labelled by a dot. In particular, classic compositions can be regarded as dotted compositions which none of its components is labelled. The vector space $\sQSym$ has a basis formed by the \emph{monomial quasisymmetric functions in superspace} $M_{\alpha}$, with $\alpha=(\alpha_1,\ldots,\alpha_k)$ a dotted composition, defined by\[M_{\alpha}=\sum_{i_1<\cdots<i_k}\theta_{i_1}^{\eta_1}\cdots\theta_{i_k}^{\eta_k}x_{i_1}^{\alpha_1}\cdots x_{i_k}^{\alpha_k},\]where each $\eta_i$ is either $0$ if $\alpha_i$ is labelled by a dot or $1$ otherwise. The product of these monomials in $\sQSym$ can be described by means a type of \emph{overlapping shuffles}. Respectively, the coproduct of a monomial is obtained by deconcatenating its components. Since monomial symmetric functions in superspace $m_{\Lambda}$ can be expanded in terms of the $M_{\alpha}$'s, then $\sSym$ is a Hopf subalgebra of $\sQSym$. See \cite{FiLaPi19} for details.

The Hopf algebra $\sNSym$ has a basis $\{H_{\alpha}\mid\alpha\text{ is a dotted composition}\}$ that is dual to the monomial quasisymmetric functions in superspace, that is, $\langle H_{\alpha},M_{\beta}\rangle=\delta_{\alpha\beta}$.

It is known that $\sNSym$ is isomorphic to the free algebra generated by the family of elements $H_n:=H_{(n)}$ and $\tH_n:=H_{(\dot{n})}$. Moreover, the coproducts of $H_n$ and $\tH_n$ are similar to the ones of $h_n$ and $\th_n$ in $\sSym$, respectively. See \cite[Proposition 6.2]{FiLaPi19}. Note that $\NSym$ is the Hopf subalgebra generated by the $H_n$'s.

In this paper, we study in detail the Hopf algebra $\sNSym$. 

In Section \ref{021}, we recall the dotted compositions and characterize one of their lattice structure. It is shown in Section \ref{022} that $\sNSym$ can be graded via the degree of a dotted composition. In particular, we deduce that the dimension of its $n$th homogeneous component is $2\cdot3^{n-1}$. By means this graduation, we explicitly describe the antipode $S:\sNSym\to\sNSym$. In this section we also show two Hopf subalgebras of $\sNSym$ that can be realised as polynomial structures.

Recall that the classic noncommutative power sum function $P_n$ is a primitive element of $\NSym$, and so is of $\sNSym$. In Subsection \ref{023}, by direct study of the coalgebra structure of $\sNSym$, we deduce the existence of a second family of primitive elements $\{\Psi_n\mid n\geq0\}$, which are different of the analogous power functions in superspace.

Similar to the classic noncommutative case, in Section \ref{024} we define families of noncommutative elementary and power sum functions in superspace $\tS_n$ and $\tP_n$, respectively. In particular, we show that $S(\tH_n)=(-1)^{n+1}\tS_n$ and that each $\tP_n$ can be written as a sum of Lie brackets $[P_i,\Psi_j]$, where $i+j=n$.

The classic noncommutative Ribbon Schur functions were introduced in \cite{GKLLRT95} via quasi determinants. In Section \ref{027}, we extend these functions to superspace. In particular we provide a formula for the product of noncommutative Ribbon Schur functions in superspace, and we study their relation with another bases of $\sNSym$.

In Section \ref{030}, we conclude the paper by showing the interaction of $\sNSym$ with other Hopf algebras. First, we extend the identification of $\NSym$ as a substructure of the noncommutative version of the Connes--Kreimer Hopf algebra \cite{Fo02,Ho03,PreFo09,Hof09}, obtaining a Hopf algebra generated by planar rooted trees. Thus, the description of the coproduct of $\sNSym$ can be given in terms of the admissible cuts of these trees. Secondly, due to the dual relation between $\sNSym$ and $\sQSym$, we show that, for every dotted composition $\alpha$, $M_{\alpha}$ can be written in terms of the so--called \emph{fundamental quasisymmetric functions in superspace}. This implies that the set of these functions $\{L_{\alpha}\mid\alpha\text{ a dotted composition}\}$ is a basis of $\sQSym$. Furthermore, we provide formulas for the coproduct of these functions, which can be described by means a new type of diagrams that extend the classic Ribbon diagrams. Finally, we study the morphism $\pi:\sNSym\to\sSym$, introduced in \cite{FiLaPi19}, obtaining new relations for symmetric functions in superspace and a new basis for $\sSym$ by means the projection of noncommutative Ribbon Schur functions in superspace on $\sSym$.

In the present paper, every linear algebraic structure will be taken over the field of rational numbers $\field$. The word \emph{algebra} will mean associative algebra with unit and the word \emph{coalgebra} will mean coassociative coalgebra with counit.

For every positive integer $n$, we will denote by $[n]$ the set $\{1,\ldots,n\}$, and set $[n]_0=\{0\}\cup[n]$. As usual, the \emph{symmetric group} on $[n]$ will be denoted by $\Sym_n$.

\section{Hopf algebras}

Here, we will recall some notions about bialgebras and Hopf algebras.

A \emph{bialgebra} is a vector space $\mH$ endowed with an associative product $m:\mH\otimes\mH\to\mH$ together with a unit $u:\field\to \mH$, and a coassociative coproduct $\Delta:\mH\to\mH\otimes\mH$ together with a counit $\epsilon:\mH\to\field$, which are compatible, that is, the coproduct and the counit are algebra morphisms. The \emph{reduced coproduct} of $\mH$ is defined by $\bar{\Delta}(x)=\Delta(x)-x\otimes1-1\otimes x$ for all $h\in\mH$. We say that $\mH$ is said to be \emph{graded} if there are subspaces $\mH_0,\mH_1,\ldots$ of $\mH$, called \emph{homogeneous components}, satisfying the conditions\[\mH=\bigoplus_{n\geq0}\mH_n,\quad\mH_i\mH_j\subseteq\mH_{i+j},\,\,i,j\geq0,\quad\Delta(\mH_n)\subseteq\bigoplus_{i+j=n}\mH_i\otimes\mH_j.\]Moreover, $\mH$ is called \emph{connected} whenever $\mH_0\simeq\Q$.

A bialgebra $\mH$, as above, is said to be a \emph{Hopf algebra} if there is a linear map $S:\mH\to\mH$, called \emph{antipode}, satisfying the relations\[m\circ(S\otimes\id)\circ\Delta=u\circ\epsilon=m\circ(\id\otimes S)\circ\Delta.\]

Note that, as $\Delta:\mH\to\mH\otimes\mH$ must be an algebra morphism, the bialgebra structure of $\mH$ depends on the product of $\mH\otimes\mH$. In particular, if $\mH$ is graded, this product may be defined as $(m\otimes m)\circ(\id\otimes\tau\otimes\id)$, where $\tau:\mH\otimes\mH\to\mH\otimes\mH$ is the \emph{twist map} defined by $\tau(x\otimes y)=(-1)^{ab}(y\otimes x)$ with $x\in\mH_a$ and $y\in\mH_b$.

An element $x$ of a Hopf algebra $\mH$ is called \emph{primitive}, if $\Delta(x)=1\otimes x+x\otimes1$ or equivalently $\bar{\Delta}(x)=0$.

It is known that if $\mH$ is a connected graded bialgebra, then it is a Hopf algebra with antipode defined recursively by $S(1)=1$ and for $x$ in some homogeneous component,\[S(x)=-x-\sum S(x_{(1)})x_{(2)},\quad\text{where}\quad\bar{\Delta}(x)=x_{(1)}\otimes x_{(2)}\quad\text{(Sweedler's notation)}.\]

We conclude this subsection by describing the dual notion of a connected graded Hopf algebra ${\mH=\displaystyle\bigoplus_{n\geq0}\mH_n}$\\[-3.5mm]with finite dimensional homogeneous components. 

First, observe that $(\mH\otimes\mH)^*\simeq\mH^*\otimes\mH^*$, and that $\mH\otimes\mH$ is also graded as a vector space. More specifically,\[\mH\otimes\mH=\bigoplus_{n\geq0}(\mH\otimes\mH)_n,\quad\text{where}\quad(\mH\otimes \mH)_n=\bigoplus_{k=0}^n\mH_k\otimes\mH_{n-k}.\]
This, together with the coproduct $\Delta$ of $\mH$ induce a product $m_d$ on ${\displaystyle\mH^*=\bigoplus_{n\geq0}\mH_n^*}$ given by the composition\[m_d:\mH^*\otimes\mH^*\stackrel{\sim\,\,\,}{\to}(\mH\otimes\mH)^*\stackrel{\Delta^*}{\to}\mH^*,\quad\text{where}\quad\Delta^*(f)(x):=f\circ\Delta(x)\quad\text{with}\quad f\in(\mH\otimes\mH)^*,\,x\in\mH.\]Respectively, the product $m$ of $\mH$ induces a coproduct $\Delta_d$ of $\mH^*$ defined as follows\[\Delta_d:\mH^*\stackrel{m^*}{\to}(\mH\otimes\mH)^*\stackrel{\sim\,\,\,}{\to}\mH^*\otimes\mH^*,\quad\text{where}\quad m^*(f)(x\otimes y):=f\circ m(x\otimes y)\quad\text{with}\quad x\otimes y\in\mH\otimes\mH,\,f\in\mH^*.\]The structure $(\mH^*,m_d,\Delta_d)$ is called the \emph{dual Hopf algebra} of $\mH$.

For more details about Hopf algebras, see, for example \cite{GrRe14,PreFo09}.

\section{Dotted compositions}\label{021}

In what follows, $\N=\{1,2,\ldots\}$ and $\dot{\N}_0=\{\dot{0},\dot{1},\ldots\}$.

A \emph{dotted composition} is a tuple $\alpha=(\alpha_1,\ldots,\alpha_k)$ whose \emph{components}, denoted by $\alpha_1,\ldots,\alpha_k$, belong to the set $\N\cup\dot{\N}_0$. A component of $\alpha$ is called \emph{dotted} if it belongs to $\dot{\N}_0$. We call $k$ the \emph{length} of $\alpha$ and denoted it by $\ell(\alpha)$. The \emph{fermionic degree} of $\alpha$ is the number $\df(\alpha)$ of dotted components of it. The \emph{degree} of $\alpha$ is the number $|\alpha|$ obtained by adding its fermionic degree with the numerical values of its components.

For instance, below the six dotted compositions of degree two:\[(1,1)\kern+1.5em(\dot{0},1)\kern+1.5em(1,\dot{0})\kern+1.5em(\dot{0},\dot{0})\kern+1.5em(2)\kern+1.5em(\dot{1})\]

For a length $k$ dotted composition $\alpha$, we will denote by $\rg(\alpha)$ the \emph{rightmost component} of $\alpha$, that is, $\rg(\alpha)=\alpha_k$. Note that $\alpha_1$ is its \emph{leftmost component}. Also, if $\alpha_1=\cdots=\alpha_k$, we denote $\alpha$ by $(\alpha_1^k)$ instead.

Graphically, dotted compositions are represented by \emph{Ribbon diagrams} that extend the classical ones, such that $m\in\N\cup\dot{\N}$ is represented by $m$ boxes, and if $m$ is dotted, a bullet is added either on the left or on the right.\begin{figure}[H]\figone\end{figure}The diagram with all its dots on the left is called the \emph{left Ribbon diagram} of the dotted composition. Similarly, we define the \emph{right Ribbon diagram}. For instance, below the left and right Ribbon diagrams of $(\dot{2},3,\dot{0},1,2,\dot{2})$.
\begin{figure}[H]\figtwo\end{figure}

\begin{pro}\label{002}
For every positive integer $n$, the number of dotted compositions of degree $n$ is $2\cdot3^{n-1}$. See \cite[A025192]{OEIS}.
\end{pro}
\begin{proof}
By considering the degree of its components, every dotted composition of degree $n$ can be regarded as a composition of $n$, where its components can be of two kinds. Since for each $k\in[n]$, there are $\binom{n-1}{k-1}$ usual compositions of degree $n$ and length $k$, then there are $\binom{n-1}{k-1}2^k$ dotted compositions of degree $n$ and length $k$. Hence, the number of dotted compositions of degree $n$ is\[\sum_{k=1}^n\binom{n-1}{k-1}2^k=\sum_{k=0}^n\binom{n-1}{k}2^{k+1}=2\sum_{k=0}^n\binom{n-1}{k}2^k=2(1+2)^{n-1}=2\cdot3^{n-1}.\]
\end{proof}

Given two dotted compositions $\alpha$ and $\beta$, we denote by $\alpha\beta$ the dotted composition obtained by concatenating $\alpha$ with $\beta$. If $\beta=(x)$ for some $x\in\N\cup\dot{\N}$, we denote $\alpha\beta$ by $\alpha x$ instead. A \emph{convex partition} of a dotted composition $\alpha$ is a tuple of dotted compositions $(\beta_1,\ldots,\beta_k)$ such that $\alpha=\beta_1\cdots\beta_k$.

In order to describe one of the partial orders on dotted compositions introduced in \cite{FiLaPi19}, we first extend the usual sum of $\N_0=\{0,1,2,\ldots\}$ to $\N_0\cup\dot{\N}_0$ as follows:\[m\oplus n=r,\kern+1.5em m\oplus\dot{n}=\dot{r},\kern+1.5em\dot{m}\oplus n=\dot{r},\kern+1.5em\dot{m}\oplus\dot{n}=0,\quad\text{where}\quad r=m+n.\]For instance,\[1\oplus1=2\qquad\dot{0}\oplus1=\dot{1}\qquad1\oplus\dot{0}=\dot{1}\qquad\dot{0}\oplus\dot{0}=0\]
Note that the length two compositions of $m\in\N\cup\N_0$, with respect to $\oplus$, are given by cutting internally its left and right Ribbon diagrams. For instance, for $2$ and $\dot{2}$, we have:\[\figfou\]

\begin{rem}\label{015}
Note that $\oplus$ is not associative in general, indeed $(1\oplus\dot{1})\oplus\dot{1}=\dot{2}\oplus\dot{1}=0$ and on the other hand $1\oplus(\dot{1}\oplus\dot{1})=1\oplus0=1$. However, if at most one of $x,y,z$ belongs to $\dot{\N}$, then $(x\oplus y)\oplus z=x\oplus(y\oplus z)$. This property will be important to describe the partial order on dotted compositions.
\end{rem}

For $\alpha,\beta$ dotted compositions of degree $n$, we say that $\beta$ \emph{covers} $\alpha$, if $\beta$ is obtained by $\oplus$-summing two consecutive components of $\alpha$ that are not dotted at once. The partial order obtained by closuring this relation is simply denoted by $\preceq$ and gives a structure of poset to the collection of all dotted compositions of degree $n$. More generally, by abuse of language, if $\alpha\preceq\beta$, we will say that $\beta$ \emph{covers} $\alpha$ as well.

Note that two dotted compositions that are comparable respect to covering have the same degree and fermionic degree. Also, a dotted composition is maximal respect to covering if it has a unique component or its components are all dotted, that is, the length and fermionic degree coincide.

\begin{pro}\label{003}
For every positive integer $n$, the number of all dotted compositions covered by $(\dot{n})$ is $(n+2)2^{n-1}$. 
\end{pro}
\begin{proof}
Let $\alpha$ be a dotted composition satisfying $\alpha\preceq(\dot{n})$. Denote by $\bar{\alpha}$ the tuple obtained by removing the zeros and the dots from the components of $\alpha$. Note that $\alpha$ has a unique component in $\dot{\N}$ and that $\bar{\alpha}$ is a composition of $n$. For $k\in[n]$, we will count the number of dotted compositions $\alpha\preceq(\dot{n})$ such that $\ell(\bar{\alpha})=k$. 

Recall that $n$ has $\binom{n-1}{k-1}$ compositions of length $k$. Then, there are $(k+1)\binom{n-1}{k-1}$ dotted compositions $\alpha\preceq(\dot{n})$ such that $\dot{0}$ is a component of $\alpha$ and $\ell(\bar{\alpha})=k$. On the other hand, there are $k\binom{n-1}{k-1}$ dotted compositions $\alpha\preceq(\dot{n})$ in which the unique component belonging to $\dot{\N}$ is not $\dot{0}$. Hence, there are $(2k+1)\binom{n-1}{k-1}$ dotted compositions $\alpha\preceq(\dot{n})$ satisfying $\ell(\bar{\alpha})=k$. Therefore, the number of dotted compositions covered by $(\dot{n})$ is\[\textstyle{\sum_{k=1}^n(2k+1)\binom{n-1}{k-1}=(n+2)2^{n-1}.}\]
\end{proof}

For instance, there are eight dotted compositions covered by $(\dot{2})$.\begin{equation}\label{001}
\figthr
\end{equation}

For a dotted composition $\alpha$, denote by $\cv(\alpha)$ the set of all dotted compositions that covers $\alpha$, that is, $\cv(\alpha)=\{\beta\mid\alpha\preceq\beta\}$. Note that $\cv(\alpha)=\{\alpha\}$ if $\alpha$ is maximal respect to covering. Moreover, if $\alpha$ is a usual composition of degree $n$, then $\cv(\alpha)$ has a maximum given by $(n)$. 

If $\df(\alpha)\geq1$, we characterize the maximal elements in $\cv(\alpha)$ by means the following proposition.

\begin{pro}
Let $\alpha$ be a dotted composition with $s:=\df(\alpha)\geq1$, and let $m$ be a maximal element in $\cv(\alpha)$. Then, each component of $m$ is obtained by $\oplus$-summing the components of a dotted composition in some convex partition $(\beta_1,\ldots,\beta_s)$ of $\alpha$ whose elements have exactly one dotted component. More specifically, $m=(\hat{\beta}_1,\ldots,\hat{\beta}_s)$, where $\hat{\beta}_i$ is the $\oplus$-sum of all components of $\beta_i$.
\end{pro}
\begin{proof}
It follows directly from the definition of the partial order $\preceq$.
\end{proof}

For instance, for $\alpha=(\dot{0},1,2,\dot{3},1,\dot{2},\dot{4})$, the set $\cv(\alpha)$ contains the following six maximal elements:
\begin{figure}[H]
\figfiv
\end{figure}

\section{The Hopf structure of $\sNSym$}\label{022}

In this section, we describe the Hopf structure of noncommutative symmetric functions in superspace $\sNSym$, introduced in \cite{FiLaPi19}.

Recall that, as in the classic case, the set of monomial quasisymmetric functions in superspace $M_{\alpha}$, with $\alpha$ a dotted composition, form a basis of the Hopf algebra $\sQSym$.

The \emph{Hopf algebra of noncommutative symmetric functions in superspace} $\sNSym$ is the dual structure of the Hopf algebra of quasisymmetric functions in superspace, given by\[\langle H_{\alpha},M_{\beta}\rangle=\delta_{\alpha\beta},\quad\text{where}\quad\alpha,\beta\quad\text{are dotted compositions}.\]

As it was shown in \cite{FiLaPi19}, the algebra structure of $\sNSym$ is freely generated by the set $\{H_r\mid r\in\N\cup\dot{\N}\}$. The generators $H_r$ will be called \emph{noncommutative complete homogeneus symmetric function in superspace}. For a length $k$ dotted composition $\alpha$, we write $H_{\alpha}=H_{\alpha_1}\cdots H_{\alpha_k}$. The set $\{H_{\alpha}\mid \alpha\text{ is a dotted composition}\}$ is a basis of $\sNSym$.

As with $\sQSym$, the algebra $\sNSym$ can be graded by the fermionic degree. With respect to this graduation, $\sNSym\otimes\sNSym$ can be endowed with the following product:\begin{equation}\label{041}(f_1\otimes g_1)\cdot(f_2\otimes g_2)=(-1)^{\df(g_1)\df(f_2)}\,f_1\,f_2\otimes g_1g_2,\quad\text{where}\quad f_1,g_1,f_2,g_2\in\sNSym.\end{equation}

Since $\sNSym$ is a free algebra, it is enough to define the coproduct $\Delta$ on its generators and so extend it as an algebra morphism, with respect to the product in \ref{041}. Thus, we have:\[\Delta(H_r)=\sum_{p\oplus q=r}H_p\otimes H_q,\quad\text{where}\quad p,q,r\in\N_0\cup\dot{\N}_0\quad\text{and}\quad H_0=1.\]

As in \cite{FiLaPi19}, for every nonnegative integer $n$, we will write $H_{\dot{n}}$ by $\tH_n$ instead. Note that $\Delta(1)=1\otimes1$ and, for every $n\in\N$, we have the following:\begin{equation}\label{004}\Delta(H_n)=\sum_{k=0}^nH_k\otimes H_{n-k},\qquad\Delta(\tH_n)=\sum_{k=0}^n(\tH_k\otimes H_{n-k}+H_{n-k}\otimes\tH_k).\end{equation}

There is another graduation of $\sNSym$ given by the degree of dotted compositions. Thus, if $\alpha$ is a dotted composition, then $H_{\alpha}$ is an homogeneous element of degree $|\alpha|$. For every $n\in\N$, we denote by $\sNSym_n$ the subspace spanned by the elements $H_{\alpha}$ such that $\alpha$ is a dotted composition of degree $n$. By Proposition \ref{002}, $\dim(\sNSym_n)=2\cdot3^{n-1}$. Fixing $\sNSym_0=K$, we have that\[\sNSym=\bigoplus_{n\geq0}\sNSym_n,\]is a connected graded bialgebra, and then a Hopf algebra. Note that $\sNSym$ is cocommutative up a sign and that $\NSym$ is a Hopf subalgebra of it.

\begin{rem}
Note that if $\alpha$ is a dotted composition, the coproduct of $H_{\alpha}$ depends both on the fermionic degree of $\alpha$ and on the action of $\Delta$ on each $H_{\alpha_i}$. In particular, if $\df(\alpha)\in\{0,1\}$, the addends of $\Delta(H_{\alpha})$ are all positive. Otherwise, some of these addends can be negative. For instance, if $\alpha=(\dot{0},1,\dot{0})$, we have\figeig
In Subsection \ref{029}, we will give formulas to compute the coproduct of $H_{\alpha}$ in terms of trees.
\end{rem}

\subsection{Polynomial Hopf subalgebras}

Here, to show how acts the fermionic degree, we present two simple Hopf subalgebras of $\sNSym$, which can be released as Hopf algebras of polynomials in one variable.

The first one is the well known Hopf subalgebra of $\sNSym$ generated by $H_1$, which is isomorphic to the classic Hopf structure of polynomials in one variable $X$, by identifying $H_1$ with this variable. It is easy to check the following\[\Delta(H_{(1^n)})=\sum_{k=0}^n\binom{n}{k}H_{(1^{n-k})}\otimes H_{(1^k)},\quad\text{where}\quad H_{(1^0)}=1.\] 

The second substructure is the Hopf subalgebra of $\sNSym$ generated by $\tH_0$. To describe the coproduct of $H_{(\dot{0}^n)}=\tH_0\cdots\tH_0$, $n$ times, for $0\leq k\leq n$ we define the following number\[\tC(n,k)=\left\{\begin{array}{ll}
1&\text{if }k=0\text{ or }k=n,\\
0&\text{if }n\text{ is even and }k\text{ is odd},\\
\tC(n-1,k-1)+\tC(n-1,k)&\text{otherwise}.
\end{array}\right.\]
Note that if $n$ is odd and $k$ is even, then\begin{equation}\label{044}\tC(n,k)=\tC(n,k+1)=\tC(n-1,k).\end{equation}Diagrammatically,\[\begin{array}{lc}
n=0\qquad&{\red1}\\
n=1&1\quad1\\
n=2&{\red1}\quad0\quad{\red1}\\
n=3&1\quad1\quad1\quad1\\
n=4&{\red1}\quad0\quad{\red2}\quad0\quad{\red1}\\
n=5&1\quad1\quad2\quad2\quad1\quad1\\
n=6&{\red1}\quad0\quad{\red3}\quad0\quad{\red3}\quad0\quad{\red1}\\
n=7&1\quad1\quad3\quad3\quad3\quad3\quad1\quad1\\
n=8&{\red1}\quad0\quad{\red4}\quad0\quad{\red6}\quad0\quad{\red4}\quad0\quad{\red1}
\end{array}\kern+3em
\begin{array}{c}
{\red1}\\
{\red1}\quad{\red1}\\
{\red1}\quad{\red2}\quad{\red1}\\
{\red1}\quad{\red3}\quad{\red3}\quad{\red1}\\
{\red1}\quad{\red4}\quad{\red6}\quad{\red4}\quad{\red1}
\end{array}\]

\begin{pro}\label{045}
We have $\tC(2n,2k)={\displaystyle\binom{n}{k}}$ for all $0\leq k\leq n$. In consequence, we obtain\[\tC(2n-1,2k)=\binom{n-1}{k}\quad\text{and}\quad\tC(2n-1,2k-1)=\binom{n-1}{k-1}.\]
\end{pro}
\begin{proof}
We proceed by induction on $n$. If $n=0$, it is clear because $k=0$ as well. If $n>0$, by applying \ref{044} and inductive hypothesis, we have\[\begin{array}{rcl}
\tC(2n,2k)&=&\tC(2n-1,2k-1)+\tC(2n-1,2k),\\
&=&\tC(2n-2,2k-2)+\tC(2n-2,2k),\\
&=&\tC(2(n-1),2(k-1))+\tC(2(n-1),2k),\\[1mm]
&=&\binom{n-1}{k-1}+\binom{n-1}{k},\\[1.5mm]
&=&\binom{n}{k}.
\end{array}\]
\end{proof}

\begin{pro}\label{046}
For $0\leq k\leq n$, we have ${\displaystyle\Delta(H_{(\dot{0}^n)})=\sum_{k=0}^n\tC(n,k)\,H_{(\dot{0}^{n-k})}\otimes H_{(\dot{0}^k)}}$, where $H_{(\dot{0}^0)}=1$.
\end{pro}
\begin{proof}
We proceed by induction on $n$. The result is obvious if $n\in\{1,2\}$. If $n\geq3$, assume that claim is true for smaller values. By inductive hypothesis, we have\[\begin{array}{rcl}
\Delta(H_{(\dot{0}^n)})&=&\Delta(H_{(\dot{0}^{n-1})})\Delta(H_{\dot{0}}),\\[2mm]
&=&{\displaystyle\left(H_{(\dot{0}^{n-1})}\otimes1+\sum_{k=1}^{n-1}\tC(n-1,k)\,H_{(\dot{0}^{n-1-k})}\otimes H_{(\dot{0}^k)}\right)\left(H_{\dot{0}}\otimes1+1\otimes H_{\dot{0}}\right)},\\[1.5mm]
&=&{\displaystyle H_{(\dot{0}^n)}\otimes1+H_{(\dot{0}^{n-1})}\otimes H_{\dot{0}}+\sum_{k=1}^{n-1}(-1)^k\tC(n-1,k)\,H_{(\dot{0}^{n-k})}\otimes H_{(\dot{0}^k)}}\\[-2mm]
&&{\displaystyle+\sum_{k=1}^{n-1}\tC(n-1,k)\,H_{(\dot{0}^{n-(k+1)})}\otimes H_{(\dot{0}^{k+1})}},\\[0.6cm]
&=&{\displaystyle H_{(\dot{0}^n)}\otimes1+1\otimes H_{(\dot{0}^n)}+(1-\tC(n-1,1))H_{(\dot{0}^{n-1})}\otimes H_{\dot{0}}}\\[2mm]
&&{\displaystyle+\sum_{k=2}^{n-1}\left[\tC(n-1,k-1)+(-1)^k\tC(n-1,k)\right]H_{(\dot{0}^{n-k})}\otimes H_{(\dot{0}^k)}}.
\end{array}\]If $n$ is even, we have $\tC(n-1,1)=\tC(n,0)=1$, and so $1-\tC(n-1,1)=0$. For $k$ odd, we have $\tC(n-1,k-1)=\tC(n-1,k)$, then $\tC(n-1,k-1)+(-1)^k\tC(n-1,k)=0$, because $(-1)^k=-1$. For $k$ even, we have $\tC(n-1,k-1)+(-1)^k\tC(n-1,k)=\tC(n-1,k-1)+\tC(n-1,k)=\tC(n,k)$.

Now, assume that $n$ is odd. If $k$ is even, by using \ref{044}, we have $\tC(n-1,k-1)+(-1)^k\tC(n-1,k)=\tC(n-1,k)=\tC(n,k)$. Finally, if $k$ is odd, we have $\tC(n-1,k-1)+(-1)^k\tC(n-1,k)=\tC(n-1,k-1)=\tC(n,k)$.
\end{proof}

By using Proposition \ref{046} and Proposition \ref{045}, we obtain
\begin{itemize}
\item${\displaystyle\Delta(H_{(\dot{0}^{2n})})=\sum_{k=0}^n\binom{n}{k}H_{(\dot{0}^{2n-2k})}\otimes H_{(\dot{0}^{2k})}}$.
\item${\displaystyle\Delta(H_{(\dot{0}^{2n-1})})=\sum_{k=0}^{n-1}\binom{n-1}{k}H_{(\dot{0}^{2n-1-2k})}\otimes H_{(\dot{0}^{2k})}+\sum_{k=1}^n\binom{n-1}{k-1}H_{(\dot{0}^{2n-2k})}\otimes H_{(\dot{0}^{2k-1})}}$.
\end{itemize}

By identifying $\tH_0$ with the variable $X$, we obtain a new structure on the set $\mP$ of polynomials in one variable. In particular, the fermionic degree inherits an algebraic structure on $\mP\otimes\mP$ given by\[(X^{m_1}\otimes X^{m_2})\cdot(X^{n_1}\otimes X^{n_2})=(-1)^{m_2n_1}X^{m_1+n_1}\otimes X^{m_2+n_2}.\]

Note that the algebraic structures presented above differ in their coalgebraic structure. For instance, the coproducts of $X^4$, with respect to the identifications above, are given respectively by\[X^4\otimes1+4X^3\otimes X+6X^2\otimes X^2+4X\otimes X^3+1\otimes X^4\quad\text{and}\quad X^4\otimes1+2X^2\otimes X^2+1\otimes X^4.\]

\subsection{Antipode}

Recall that the antipode of a connected graded bialgebra $\mH$ is defined recursively by conditions $S(1)=1$ and\[S(x)=-\left(x+\sum S(x_{(1)})x_{(2)}\right),\quad\text{where}\quad x\in\mH\quad\text{and}\quad\bar{\Delta}(x)=\sum x_{(1)}\otimes x_{(2)}.\]

For a dotted composition $\alpha$, denote by $\mathrm{rev}(\alpha)$ the dotted composition obtained by reversing its components.

It is well known that \cite[Remark 5.3.4]{GrRe14}, for a usual composition $\alpha$, we have\begin{equation}\label{042}S(H_{\alpha})=\sum_{\beta\preceq\mathrm{rev}(\alpha)}(-1)^{\ell(\beta)}H_{\beta}.\end{equation}
A similar result can be generalized for dotted compositions up a sign that depends on the fermionic degree. Previously, we determine the antipode of generators indexed by dotted compositions of length one.

\begin{pro}\label{000}
For every integer $n\geq0$, we have ${\displaystyle S(\tH_n)=\sum_{\alpha\preceq(\dot{n})}(-1)^{\ell(\alpha)}H_{\alpha}}$.
\end{pro}
\begin{proof}
We proceed by induction on $n$. Since $\bar{\Delta}(\tH_0)=0$, $S(\tH_0)=-\tH_0$. Consider $n\geq1$ and assume the claim is true for all $k<n$. Since $\bar{\Delta}(\tH_n)=\sum_{k=0}^{n-1}(\tH_k\otimes H_{n-k}+H_{n-k}\otimes\tH_k)$, we obtain the following formula\[S(\tH_n)=-\left(\tH_n+\sum_{k=0}^{n-1}(S(\tH_k)H_{n-k}+S(H_{n-k})\tH_k)\right).\]By rewriting ${\displaystyle\sum_{k=0}^{n-1}S(\tH_k)H_{n-k}}$ as ${\displaystyle\sum_{k=1}^nS(\tH_{n-k})H_k}$ and applying the inductive hypothesis, we obtain\[S(\tH_n)=-\tH_n+\sum_{k=1}^n\sum_{\beta\preceq(\dot{m})}(-1)^{\ell(\beta)+1}H_{\beta}H_k+\sum_{k=0}^{n-1}\sum_{\beta\preceq(m)}(-1)^{\ell(\beta)+1}H_{\beta}\tH_k,\]where $m=n-k$. On the other hand, if $\alpha\preceq(\dot{n})$, then either $H_{\alpha}=H_{\beta}\tH_k$ for some $k\in[n]$ and $\beta\preceq(m)$, or $H_{\alpha}=H_{\beta}H_k$ for some $k\in[n]_0$ and $\beta\preceq(\dot{m})$. Hence, the result follows.
\end{proof}

For instance, from the lattice in \ref{001}, we obtain\fignin

\begin{pro}\label{043}
For a dotted composition $\alpha$ be a dotted composition, we have\[S(H_{\alpha})=(-1)^{\binom{\df(\alpha)}{2}}\sum_{\beta\preceq\mathrm{rev}(\alpha)}(-1)^{\ell(\beta)}H_{\beta}.\]
\end{pro}
\begin{proof}
Let $\alpha$ be a dotted composition with $k=\ell(\alpha)$, and, for each $i\in[k]$, let $a_i$ be the fermionic degree of $(\alpha_i)$. Note that $a_i=1$ if $\alpha_i\in\dot{\N}$ and $a_i=0$ otherwise. If $k=1$, the result follows from \ref{042} or Proposition \ref{000}. Now, assume that $k\geq2$. Since the antipode is a signed antihomomorphism, Proposition \ref{000} implies that\[S(H_{\alpha})=(-1)^aS(H_{\alpha_k})\cdots S(H_{\alpha_1})=(-1)^a\sum_{\beta\preceq\mathrm{rev}(\alpha)}(-1)^{\ell(\beta)}H_{\beta},\quad\text{where}\quad a=\sum_{i=1}^{k-1}a_i(a_{i+1}+\cdots+a_k).\]Now, we will show that $a=\binom{\df(\alpha)}{2}$. Let $\alpha_{i_1},\ldots,\alpha_{i_s}$, with $i_1<\cdots<i_s$, be the dotted components of $\alpha$, that is, $s=\df(\alpha)$. We have $a_{i_1}=\cdots=a_{i_s}=1$ and $a_i=0$ otherwise. This implies that\[a=\sum_{j=1}^{s-1}a_{i_j}(a_{i_{j+1}}+\cdots+a_{i_s})=\sum_{j=1}^{s-1}(s-j)=s(s-1)-\frac{(s-1)s}{2}=\frac{(s-1)s}{2}=\binom{s}{2}=\binom{\df(\alpha)}{2}.\]
\end{proof}
Note that Proposition \ref{043} is dual to \cite[Proposition 5.10]{FiLaPi19} for the antipode of $M_{\alpha}$ in $\sQSym$.

\subsection{Primitive elements}\label{023}

In this section we study the primitive part of $\sNSym$. Recall that the classic noncommutative power sum functions $P_n$ are defined recursively by $P_1=H_1$ and\begin{equation}\label{020}\sum_{k=0}^{n-1}H_nP_{n-k}=nH_n,\quad n\geq1.\end{equation}It is known that $P_n$ is a primitive element of the Hopf algebra $\NSym$, and that the primitive part of $\NSym$ is the free Lie algebra generated by these functions \cite[Subsection 3.1]{GKLLRT95}. Note that, as $\NSym$ is a Hopf subalgebra of $\sNSym$, then $P_n$ is a primitive element of $\sNSym$ as well.

Here, we determined another primitive elements of $\sNSym$, which are defined recursively by $\Psi_0=\tH_0$ and\begin{equation}\label{019}\Psi_n=\tH_n-\sum_{k=0}^{n-1}H_{n-k}\Psi_k,\quad n\geq1.\end{equation}

\begin{pro}\label{012}
For every $n\geq0$, $\Psi_n$ is a primitive of $\sNSym$.
\end{pro}
\begin{proof}
By definition, $\Psi_0$ is a primitive, so we proceed by induction on $n$. Assume the claim is true for all $k<n$, that is $\Delta(\Psi_k)=\Psi_k\otimes1+1\otimes\Psi_k$. Since $\Delta(H_{n-k})=\sum_{k=0}^{n-k}H_i\otimes H_{n-k-1}$, we obtain\[\Delta(H_{n-k})\Delta(\Psi_k)=\left(\sum_{i=0}^{n-k}H_i\otimes H_{n-k-i}\right)\left(\Psi_k\otimes1+1\otimes\Psi_k\right)=\sum_{i=0}^{n-k}H_i\Psi_k\otimes H_{n-k}+\sum_{i=0}^{n-k}H_i\otimes H_{n-k}\Psi_k.\]Hence,\[\begin{array}{rcl}
\Delta(\Psi_n)&=&{\displaystyle\Delta(\tH_n)-\sum_{k=0}^{n-1}\Delta(H_{n-k})\Delta(\Psi_k),}\\[0.5cm]
&=&{\displaystyle\sum_{k=0}^n(\tH_k\otimes H_{n-k}+H_{n-k}\otimes\tH_k)-\sum_{k=0}^{n-1}\sum_{i=0}^{n-k}(H_i\Psi_k\otimes H_{n-k-i}+H_{n-k-i}\otimes H_i\Psi_k)},\\[0.5cm]
&=&{\displaystyle\sum_{i=1}^n\left(\tH_{n-i}\otimes H_i+H_i\otimes\tH_{n-i}\right)+\tH_n\otimes1+1\otimes\tH_n-\left(\sum_{k=0}^{n-1}H_{n-k}\Psi_k\right)\otimes1-1\otimes\left(\sum_{k=0}^{n-1}H_{n-k}\Psi_k\right)}\,\,\\[0.5cm]
&&-{\displaystyle\sum_{i=1}^n\left[\left(\sum_{k=0}^{n-i}H_{n-i-k}\Psi_k\right)\otimes H_i+H_i\otimes\left(\sum_{k=0}^{n-i}H_{n-i-k}\Psi_k\right)\right]},\\[0.5cm]

&=&{\displaystyle\sum_{i=1}^n\left[\left(\tH_{n-i}-\sum_{k=0}^{n-i}H_{n-i-k}\Psi_k\right)\otimes H_i+H_i\otimes\left(\tH_{n-i}-\sum_{k=0}^{n-i}H_{n-i-k}\Psi_k\right)\right]}\\[0.5cm]
&&+{\displaystyle\left(\tH_n-\sum_{k=0}^{n-1}H_{n-k}\Psi_k\right)\otimes1+1\otimes\left(\tH_n-\sum_{k=0}^{n-1}H_{n-k}\Psi_k\right)}.
\end{array}\]
Therefore $\Delta(\Psi_n)=\Psi_n\otimes1+1\otimes\Psi_n$.
\end{proof}

Recall that $P_n$ can be written as\begin{equation}\label{007}P_n=\sum_{\alpha\preceq(n)}(-1)^{\ell(\alpha)-1}\rg(\alpha)H_{\alpha}.\end{equation}

Now, by an inductive argument, we have the following proposition for $\Psi_n$.
\begin{pro}\label{010}
For every integer $n\geq1$, $\Psi_n$ is the sum of $2^n$ terms given as\[\Psi_n=\tH_n+\sum_{k=0}^{n-1}\sum_{\alpha\preceq(n-k)}(-1)^{\ell(\alpha)}H_{\alpha}\tH_k.\]
\end{pro}

For instance, $\Psi_3$ is given by the following sum of eight terms.\begin{figure}[H]\figele\end{figure}
Recall that if $\alpha$ is a usual composition of length $k$, then $P_{\alpha}=P_{\alpha_1}\cdots P_{\alpha_k}$. It is known that\[H_n=\sum_{\alpha\preceq(n)}\frac{1}{\nu_{\alpha}}P_{\alpha},\quad\text{where}\quad\nu_{\alpha}=\alpha_1(\alpha_1+\alpha_2)\cdots(\alpha_1+\cdots+\alpha_k).\]This together with Equation \ref{019} imply the following\[\tH_n=\Psi_n+\sum_{k=0}^{n-1}\left(\sum_{\alpha\preceq(n-k)}\frac{1}{\nu_{\alpha}}P_{\alpha}\right)\Psi_k.\]

\begin{rem}
Unlike the classic notion of Hopf algebra, if $x,y$ are primitive elements of $\sNSym$, the Lie bracket $[x,y]$ is not necessarily primitive. Indeed,\[\begin{array}{rcl}\Delta([x,y])&=&(x\otimes1+1\otimes x)(y\otimes1+1\otimes y)-(y\otimes1+1\otimes y)(x\otimes1+1\otimes x),\\
&=&[x,y]\otimes1+1\otimes[x,y]+[(-1)^{\df(x)\df(y)}-1]y\otimes x+[1-(-1)^{\df(x)\df(y)}]x\otimes y.
\end{array}\]Thus, $[x,y]$ is primitive if and only if $\df(x)$ or $\df(y)$ is even.

Also note that if we consider the usual product on $\sNSym\otimes\sNSym$, instead of the one in \ref{041}, we obtain a classic structure of Hopf algebra on $\sNSym$, and so the primitive part of $\sNSym$ would be the free Lie algebra generated by the set $\{P_n\mid n\in\N\}\cup\{\Psi_n\mid n\in\N_0\}$.
\end{rem}

\section{Noncommutative elementary and power sum functions in superspace}\label{024}

In this section, we introduce noncommutative analogue versions of elementary and power sum functions in superspace, defined in \cite{DeLaMa06}, by means generating functions.

In what follows, we write the generating function of the complete homogeneous functions $H_n$ and $\tH_n$, by\[\lambda(t,\tau):=\sum_{n\geq 0} t^n(H_n+\tau\tH_n),\]where $t$ and $\tau$ are indeterminate parameters.

\subsection{Elementary functions}\label{025}

The \emph{elementary functions in superspace} are defined by means the following generating function\[\sigma(t,\tau):=\sum_{n\geq 0}t^n(S_n+\tau\tS_n)\quad\text{satisfying}\quad\sigma(t,\tau)\lambda(-t,-\tau)=1.\]

\begin{pro}
For $n\geq1$, we obtain the following recursive formulas:
\[\sum_{k=0}^n(-1)^{n-k}H_kS_{n-k}=\sum_{k=0}^n(-1)^{n-k}S_{n-k}H_{n-k}=0\quad\text{and}\quad\sum_{k=0}^n(-1)^{n-k}(S_k\tH_{n-k}-\tS_kH_{n-k})=0.\]
\end{pro}
Note that $S_n$ is the classic noncommutative elementary function and\[\tS_n=S_n\tH_0+\sum_{k=0}^{n-1}(-1)^{n-k}(S_k\tH_{n-k}-\tS_kH_{n-k}).\]

\begin{pro}\label{009}
For every integer $n\geq0$, we have $S(\tH_n)=(-1)^{n+1}\tS_n$. As a consequence,\[\tS_n=(-1)^{n+1}\sum_{\alpha\preceq(\dot{n})}(-1)^{\ell(\alpha)}H_{\alpha}.\]
\end{pro}
\begin{proof}
We proceed by induction on $n$. Note that $S(\tH_0)=-\tS_0$ because $\tS_0=\tH_0$. If $n\geq1$, assume the claim is true for all $k<n$. By rewriting the coproduct of $\tH_n$ as\[\Delta(\tH_n)=\sum_{k=0}^nH_k\otimes\tH_{n-k}+\sum_{k=0}^n\tH_k\otimes H_{n-k},\]we obtain\[S(\tH_n)=-\left(\tH_n+\sum_{k=1}^nS(H_k)\tH_{n-k}+\sum_{k=0}^{n-1}S(\tH_k)H_{n-k}\right).\]The fact that $S(H_k)=(-1)^kS_k$ and the inductive hypothesis imply that\[\begin{array}{rcl}
S(\tH_n)&=&{\displaystyle-\left(\tH_n+\sum_{k=1}^n(-1)^kS_k\tH_{n-k}+\sum_{k=0}^{n-1}(-1)^{k+1}\tS_kH_{n-k}\right),}\\[1mm]

&=&{\displaystyle-\left((-1)^nS_n\tH_0+\sum_{k=0}^{n-1}(-1)^k\left(S_k\tH_{n-k}-\tS_kH_{n-k}\right)\right),}\\[1mm]

&=&{\displaystyle(-1)^{n+1}\left(S_n\tH_0+\sum_{k=0}^{n-1}(-1)^{n-k}\left(S_k\tH_{n-k}-\tS_kH_{n-k}\right)\right),}\\[1mm]

&=&(-1)^{n+1}\tS_n.
\end{array}\]Finally, Proposition \ref{000} implies that ${\displaystyle\tS_n=(-1)^{n+1}\sum_{\alpha\preceq(\dot{n})}(-1)^{\ell(\alpha)}H_{\alpha}}$.
\end{proof}

\begin{pro}\label{011}
For every $n\geq0$, $\tS_n={\displaystyle\sum_{k=0}^n(-1)^{n-k}\Psi_{n-k}S_k}$.
\end{pro}
\begin{proof}
It is a consequence of Proposition \ref{009}, Proposition \ref{010} and the fact that ${\displaystyle S_k=(-1)^k\sum_{\beta\preceq(k)}(-1)^{\ell(\beta)}H_{\beta}}$.
\end{proof}

\begin{pro}
For every $n\geq0$, we have\[\Delta(\tS_n)=\sum_{k=0}^n\left(\tS_k\otimes S_{n-k}+S_{n-k}\otimes\tS_k\right).\]In consequence, $S(\tS_n)=(-1)^{n+1}\tH_n$.
\end{pro}
\begin{proof}
By Proposition \ref{011}, we have\[\tS_n=\sum_{q+r=n}(-1)^q\Psi_qS_r,\qquad q,r\geq0.\]Recall that $\Psi_q$ is primitive (Proposition \ref{012}), that is, $\Delta(\Psi_q)=\Psi_q\otimes1+1\otimes\Psi_q$. Thus,\[\begin{array}{rcl}
\Delta(\tS_n)&=&{\displaystyle\sum_{i+j+k=n}\left[(-1)^i\Psi_iS_j\otimes S_k+S_k\otimes(-1)^i\Psi_iS_j\right],}\\[5mm]
&=&{\displaystyle\sum_{k=0}^n\left[\left(\sum_{i=0}^k(-1)^{k-i}\Psi_{k-i}S_i\right)\otimes S_{n-k}+S_{n-k}\otimes\left(\sum_{i=0}(-1)^{k-i}\Psi_{k-i}S_i\right)\right],}\\[5mm]
&=&{\displaystyle\sum_{k=0}^n\left(\tS_k\otimes S_{n-k}+S_{n-k}\otimes\tS_k\right).}
\end{array}\] 
\end{proof}

\subsection{Power sum functions}\label{026}

The \emph{power sum functions in superspace}, formed by the classic $P_n$ and its analogue in superspace $\tP_n$, are obtained by means the following generating function\[\Pi(t,\tau):=\sum_{n\geq0}t^n(P_n+\tau(n+1)\tP_n)\quad\text{satisfying}\quad\lambda(t,\tau)\Pi(t,\tau)=(t\partial_t+\tau\partial_\tau)\lambda(t,\tau).\]Due to this relation, we obtain Equation \ref{020} and the following proposition.
\begin{pro}\label{006}
For every $n\geq0$, we have\begin{equation}\label{037}\sum_{k=0}^n\left(\tH_{n-k}P_k+(k+1)H_{n-k}\tP_k\right)=(n+1)\tH_n,\quad\text{where}\quad P_0=0.\end{equation}
\end{pro}

\begin{pro}
For every $n\geq0$, we have
\[(n+1)\tP_n=\sum_{\alpha\preceq\,\dot{n}}(-1)^{\ell(\alpha)-1}|\rg(\alpha)|H_{\alpha},\]where $|\rg(\alpha)|$ denotes the degree of dotted composition $(\rg(\alpha))$.
\end{pro}
\begin{proof}
We proceed by induction on $n$. Proposition \ref{006} implies that $\tP_0=\tH_0$. Let $n\geq1$ and assume the claim is true for $0\leq k<n$. By Proposition \ref{006}, we have:\[(n+1)\tP_n=(n+1)\tH_n-\sum_{k=1}^n\tH_{n-k}P_k-\sum_{k=0}^{n-1}H_{n-k}(k+1)\tP_k.\]The results follows by replacing $P_k$ as in \ref{007} and applying the inductive hypothesis to $(k+1)\tP_k$. 
\end{proof}

\begin{pro}\label{005}
$\tP_0=\Psi_0$, and for every $n\geq1$, we have\begin{equation}\label{038}(n+1)\tP_n=(n+1)\Psi_n+\sum_{k=0}^{n-1}[P_{n-k},\Psi_k].\end{equation}
\end{pro}
\begin{proof}
As $\tH_1=H_1\Psi_0+\Psi_1$, Proposition \ref{006} implies that $\tP_0=\tH_0=\Psi_0$ and $2\tP_1=2\Psi_1+[P_1,\Psi_0]$. Now, we proceed by induction by assuming the claim is true for all $1<k<n$. By Proposition \ref{006}, we have\[(n+1)\tP_n=(n+1)\tH_n-H_n\Psi_0-\sum_{k=1}^{n-1}H_{n-k}(k+1)\tP_k-\sum_{k=1}^n\tH_{n-k}P_k.\]Inductive hypothesis implies that, for each $1\leq k<n$, we have\[(k+1)\tP_k=(k+1)\Psi_k+{\displaystyle\sum_{i=0}^{k-1}[P_{k-i},\Psi_i]}.\]Since $\displaystyle{(n+1)\tH_n=(n+1)\sum_{k=0}^nH_{n-k}\Psi_k}$, then\[(n+1)\tP_n=(n+1)\Psi_n+\sum_{k=0}^{n-1}(n-k)H_{n-k}\Psi_k-\sum_{k=1}^{n-1}\left(H_{n-k}\sum_{i=0}^{k-1}[P_{k-i},\Psi_i]\right)-\sum_{k=1}^n\tH_{n-k}P_k.\]Now, the fact that $(n-k)H_{n-k}=\displaystyle{\sum_{j=1}^{n-k}H_{n-k-j}P_j}$ implies the following\[\begin{array}{rcl}
\displaystyle{\sum_{k=0}^{n-1}(n-k)H_{n-k}\Psi_k}&=&\displaystyle{\sum_{k=0}^{n-1}\left(\sum_{j=1}^{n-k}H_{n-k-j}P_j\right)\Psi_k},\\
&=&\displaystyle{\sum_{k=0}^{n-1}P_{n-k}\Psi_k+\sum_{k=1}^{n-1}H_k(P_{n-k}\Psi_0+\cdots+P_1\Psi_{n-k-1}).}
\end{array}
\]On the other hand, we have\[\begin{array}{rcl}
\displaystyle{\sum_{k=1}^n\tH_{n-k}P_k}&=&\displaystyle{\sum_{k=1}^n\left(\sum_{j=0}^{n-k}H_{n-k-j}\Psi_j\right)P_k,}\\
&=&\displaystyle{\sum_{k=0}^{n-1}\Psi_kP_{n-k}+\sum_{k=1}^{n-1}H_k(\Psi_0P_{n-k}+\cdots+\Psi_{n-k-1}P_1).}
\end{array}\]The results above implies that\[\begin{array}{rcl}
\displaystyle{(n+1)\tP_n}&=&\displaystyle{(n+1)\Psi_n+\sum_{k=0}^{n-1}[P_{n-k},\Psi_k]+\sum_{k=1}^{n-1}H_k\left(\sum_{i=0}^{n-k-1}[P_{n-k-i},\Psi_i]\right)-\sum_{k=1}^{n-1}\left(H_{n-k}\sum_{i=0}^{k-1}[P_{k-i},\Psi_i]\right),}\\[5mm]
&=&\displaystyle{(n+1)\Psi_n+\sum_{k=0}^{n-1}[P_{n-k},\Psi_k].}
\end{array}\]
\end{proof}

Since $\df(P_{n-k})=0$, we have that $[P_{n-k},\Psi_k]$ is a primitive element. As an immediate consequence we have the following result.
\begin{crl}
$\tP_n$ is a primitive element of $\sNSym$ for all $n\geq0$.
\end{crl}

In the following proposition, we write $\Psi_n$ in terms of the $P_i$'s and $\tP_i$'s. To simplify the notation, for $g_1,\ldots,g_n$ in a Lie algebra, we set
\[[g_1,\ldots,g_n]=\left\{\begin{array}{ll}
g_1&\text{if }n=1,\\
{[g_1,[g_2,\ldots,g_n]]}&\text{if }n>1\end{array}\right.\]

\begin{pro}\label{008}
For every $n\geq0$, we have\[\Psi_n=\sum_{\alpha=(\alpha_1,\ldots,\alpha_k)}\left(\frac{(-1)^{\ell(\alpha)+1}}{t_{\alpha}}[P_{\alpha_1},\ldots,P_{\alpha_{k-1}},\tP_{\alpha_k}]+\frac{(-1)^{\ell(\alpha)}}{(\alpha_k+1)t_{\alpha}}[P_{\alpha_1},\ldots,P_{\alpha_k},\tP_0]\right),\]where $\alpha$ is a usual composition of $n$, and $\displaystyle{t_{\alpha}=\prod_{i=1}^{k-1}(\alpha_i+\cdots+\alpha_k+1)}$.
\end{pro}
\begin{proof}
The proof follows by applying Proposition \ref{005} and an inductive argument on $n$. 
\end{proof}

As a consequence, we obtain the following result.
\begin{crl}
The Lie algebras generated by $P_i,\Psi_i$ and $P_i,\tP_i$ respectively, coincide. 
\end{crl}

\section{Ribbon Schur functions in superspace}\label{027}

In this section, we extend the classic noncommutative Ribbon Schur functions \cite{GKLLRT95} to superspace. These functions form a new basis of the Hopf algebra $\sNSym$. We give an explicit formula for the product of noncommutative Ribbon Schur functions in superspace and write other basis of $\sNSym$ in terms of these functions.

Let $\alpha$ be a dotted composition. The \emph{noncommutative Ribbon Schur function in superspace} $R_{\alpha}$ is defined inductively as follows\[R_{\alpha}=H_{\alpha}-\sum_{\alpha\prec\beta}R_{\beta}.\]Notice that $R_{\alpha}=H_{\alpha}$ whenever $\alpha$ is maximal. Furthermore, if $\alpha$ is a usual composition, $R_{\alpha}$ is a classic noncommutative Ribbon Schur function.

To characterize the product of ribbon functions in superspace, we introduce a new (partial) operation $\odot$ on dotted compositions. Given $\alpha=(\alpha_1,\ldots,\alpha_r)$ and $\beta=(\beta_1,\ldots,\beta_s)$ be two dotted compositions such that $\alpha_r$ and $\beta_1$ are not dotted at once, we define $\alpha\odot\beta=(\alpha_1,\ldots,\alpha_{r-1},\alpha_r\oplus\beta_1,\beta_2,\ldots,\beta_s)$. For instance,\figten In particular, due to Remark \ref{015}, we obtain the following identities relating $\odot$ and the usual concatenation:\[(\alpha\odot\beta)\odot x=\alpha\odot(\beta\odot x),\quad(\alpha\beta)\odot x=\alpha(\beta\odot x),\quad(\alpha\odot\beta)x=\alpha\odot(\beta x),\qquad x\in\N\cup\dot{\N}.\]

The following is the main result of this section.
\begin{thm}\label{018}
Let $\alpha,\beta$ be two dotted compositions. Then\[R_{\alpha}R_{\beta}=\left\{\begin{array}{ll}
R_{\alpha\beta}&\text{if }\rg(\alpha),\beta_1\in\dot{\N}_0,\\
R_{\alpha\beta}+R_{\alpha\odot\beta}&\text{otherwise}.
\end{array}\right.\]
\end{thm}

To prove Theorem \ref{018}, first we need to show Lemma \ref{013}, Lemma \ref{014} and Lemma \ref{017}.

\begin{lem}\label{013}
Let $\alpha$ be a dotted composition, and let $x\in\N\cup\dot{\N}_0$.
\begin{enumerate}
\item If $x\in\N$, then $\cv(\alpha x)=\{\beta x\mid\alpha\preceq\beta\}\cup\{\beta\odot x\mid\alpha\preceq\beta\}$.
\item If $x\in\dot{\N}$, then $\cv(\alpha x)=\{\beta x\mid\alpha\preceq\beta\}$ whenever $\rg(\alpha)$ is dotted. Otherwise, we have $\cv(\alpha x)=\{\beta x\mid\alpha\preceq\beta\}\cup\{\beta\odot x\mid\alpha\preceq\beta,\,\rg(\beta)\in\N\}$.
\end{enumerate}
\end{lem}
\begin{proof}
By definition, $\cv(\alpha x)=\{\gamma\mid\alpha x\preceq\gamma\}$. Note that, in all cases, the set on the right is contained in $\cv(\alpha x)$. Conversely, for $\gamma\in\cv(\alpha x)$, we distinguish three cases.

Let $x\in\N$. If $\rg(\gamma)=x$, the other components of $\gamma$ are obtained as $\oplus$-sums of consecutive components of $\alpha$. So, $\gamma=\beta x$ for some $\beta\succeq\alpha$. Now, if $\rg(\gamma)\neq x$, there is $i$ such that $\rg(\gamma)=\alpha[i]\oplus x$, where $\alpha[i]$ is the $\oplus$-sum of the last $\ell(\alpha)-i-1$ components of $\alpha$. Note that at most one of the components that define $\alpha[i]$ is dotted. Consider $\gamma'$ such that $\gamma=\gamma'\rg(\gamma)$ and define the dotted composition $\beta=\gamma'\alpha[i]$. Since $\gamma'$ is obtained by covering the dotted composition formed by the the first $i$ components of $\alpha$, we obtain that $\beta\succeq\alpha$ and $\gamma=\beta\odot x$.

If $x,\rg(\alpha)\in\dot{\N}$, they are not $\oplus$-summed when $\alpha x$ is covered. Then $\gamma=\beta x$ for some dotted composition $\beta\succeq\alpha$.

Finally, consider $x\in\dot{\N}$ and $\rg(\alpha)\in\N$. As in the first case, if $\rg(\gamma)=x$, there is $\beta\succeq\alpha$ such that $\gamma=\beta x$. Now, if $\rg(\gamma)\neq x$, as above, $\gamma=\beta\odot x$, where $\beta=\gamma'\alpha[i]$ with $\alpha[i]\in\N$ because $x$ is already dotted.
\end{proof}

Notice that, from the previous lemma, for $x,y\in\N\cup\dot{\N}$, we have\begin{equation}\label{016}
\cv((x,y))=\left\{\begin{array}{ll}
\{(x,y),(x\oplus y)\}&\text{if }\{x,y\}\not\subset\dot{\N}_0,\\
\{(x,y)\}&\text{otherwise}.
\end{array}\right.
\end{equation}

\begin{lem}\label{014}
Let $\alpha$ be a dotted composition, and let $x\in\N\cup\dot{\N}$. Then\[R_{\alpha}H_x=\left\{\begin{array}{ll}R_{\alpha x}&\text{if }\rg(\alpha),x\in\dot{\N}_0,\\
R_{\alpha x}+R_{\alpha\odot x}&\text{otherwise}.\end{array}\right.\]
\end{lem}
\begin{proof}
We proceed by induction on the length of $\alpha$. If $\ell(\alpha)=1$, it is a consequence of \ref{016}. Now, if $\ell(\alpha)>1$, assume the result is true for all dotted compositions of smaller length. Recall that if $\alpha\prec\beta$, then $\ell(\beta)<\ell(\alpha)$, and that, by definition, we have\[R_{\alpha}=H_{\alpha}-\sum_{\alpha\prec\beta}R_{\beta},\quad\text{then}\quad R_{\alpha}H_x=H_{\alpha x}-\sum_{\alpha\prec\beta}R_{\beta}H_x=\sum_{\alpha x\preceq\gamma}R_{\gamma}-\sum_{\alpha\prec\beta}R_{\beta}H_x.\]We distinguish two cases. If $x\in\N$, Lemma \ref{013} implies that $\cv(\alpha x)=\{\beta x\mid\alpha\preceq\beta\}\cup\{\beta\odot x\mid\alpha\preceq\beta\}$. This together with the inductive hypothesis implies that\[R_{\alpha}H_x=\sum_{\alpha\preceq\beta}(R_{\beta x}+R_{\beta\odot x})-\sum_{\alpha\prec\beta}(R_{\beta x}+R_{\beta\odot x})=R_{\alpha x}+R_{\alpha\odot x}.\]Let $x\in\dot{\N}_0$. If $\rg(\alpha)\in\dot{\N}_0$, then $\rg(\beta)\in\dot{\N}_0$ for all $\beta\succeq\alpha$, and, by Lemma \ref{013}, we have $\cv(\alpha x)=\{\beta x\mid\alpha\preceq\beta\}$. By inductive hypothesis, we obtain\[R_{\alpha}H_x=\sum_{\alpha\preceq\beta}R_{\beta x}-\sum_{\alpha\prec\beta}R_{\beta x}=R_{\alpha x}.\]Now, if $\rg(\alpha)\in\N$, since $\{\beta\succeq\alpha\mid\rg(\beta)\in\N\}$ is nonempty, Lemma \ref{013} implies that\[H_{\alpha x}=\sum_{\alpha\preceq\beta}R_{\beta x}+\sum_{\alpha\preceq\beta,\,\rg(\beta)\in\N}R_{\beta\odot x}.\]On the other hand, by applying the inductive hypothesis for $\beta\succ\alpha$, we have $R_{\beta}H_x=R_{\beta x}$ if $\rg(\beta)\in\dot{\N}_0$, and $R_{\beta}H_x=R_{\beta x}+R_{\beta\odot x}$ if $\rg(\beta)\in\N$. Hence,\[R_{\alpha}H_x=\sum_{\alpha\preceq\beta}R_{\beta x}+\sum_{\alpha\preceq\beta,\,\rg(\beta)\in\N}R_{\beta\odot x}-\sum_{\alpha\prec\beta}R_{\beta x}-\sum_{\alpha\prec\beta,\,\rg(\beta)\in\N}R_{\beta\odot x}=R_{\alpha x}+R_{\alpha\odot x}.\]This concludes the proof.
\end{proof}

\begin{lem}\label{017}
Let $\alpha$ be a dotted composition, and let $\beta=(x,y)$ for some $x,y\in\N\cup\dot{\N}_0$. Then\[R_{\alpha}R_{\beta}=\left\{\begin{array}{ll}
R_{\alpha\beta}&\text{if }\rg(\alpha),x\in\N_0,\\
R_{\alpha\beta}+R_{\alpha\odot\beta}&\text{otherwise}.
\end{array}\right.\]
\end{lem}
\begin{proof}
Notice that, by Lemma \ref{014}, we have\[R_{\beta}=\left\{\begin{array}{ll}
R_xR_y&\text{if }x,y\in\dot{\N}_0,\\
R_xR_y-R_{x\odot y}&\text{otherwise}.
\end{array}\right.\]First, we assume that $\rg(\alpha),x\in\dot{\N}_0$. If $y\in\dot{\N}_0$, then\[R_{\alpha}R_{\beta}=R_{\alpha}(R_xR_y)=(R_{\alpha}R_x)R_y=R_{\alpha x}R_y=R_{\alpha\beta}.\]On the other hand, if $y\in\N$, then\[\begin{array}{rcl}R_{\alpha}R_{\beta}
&=&R_{\alpha}(R_xR_y-R_{x\odot y}),\\
&=&(R_{\alpha}R_x)R_y-R_{\alpha}R_{x\odot y},\\
&=&R_{\alpha x}R_y-R_{\alpha(x\odot y)},\\
&=&R_{\alpha xy}+R_{(\alpha x)\odot y}-R_{\alpha(x\odot y)},\\
&=&R_{\alpha\beta},\end{array}\]because $x\odot y\in\dot{\N}_0$ and $(\alpha x)\odot y=\alpha(x\odot y)$.

Now, we assume that $\{\rg(\alpha),x\}\not\subset\dot{\N}_0$. We will distinguish several cases.

In general, if $x\in\N$, we have\[\begin{array}{rcl}
R_{\alpha}R_{\beta}&=&R_{\alpha}(R_xR_y-R_{x\odot y}),\\
&=&(R_{\alpha}R_x)R_y-R_{\alpha}R_{x\odot y},\\
&=&(R_{\alpha x}+R_{\alpha\odot x})R_y-R_{\alpha}R_{x\odot y},\\
&=&R_{\alpha x}R_y+R_{\alpha\odot x}R_y-R_{\alpha}R_{x\odot y},\\
&=&R_{\alpha\beta}+R_{(\alpha x)\odot y}+R_{\alpha\odot x}R_y-R_{\alpha}R_{x\odot y}.
\end{array}\]This equation depends on the set the elements $\rg(\alpha)$ and $y$ belong to. If $\rg(\alpha)\in\N$, then $\rg(\alpha\odot x)\in\N$, hence\[R_{\alpha}R_{\beta}=R_{\alpha\beta}+R_{(\alpha x)\odot y}+R_{(\alpha\odot x)y}+R_{(\alpha\odot x)\odot y}-R_{\alpha(x\odot y)}-R_{\alpha\odot(x\odot y)}=R_{\alpha\beta}+R_{\alpha\odot\beta}\]because $(\alpha x)\odot y=\alpha(x\odot y)$ and $\alpha\odot(x\odot y)=(\alpha\odot x)\odot y$. Now, consider $\rg(\alpha)\in\dot{\N}_0$. Notice that $\rg(\alpha\odot x)$ belongs to $\dot{\N}_0$ as well. If $y\in\N$, then\[R_{\alpha}R_{\beta}=R_{\alpha\beta}+R_{(\alpha x)\odot y}+R_{(\alpha\odot x)y}+R_{(\alpha\odot x)\odot y}-R_{\alpha(x\odot y)}-R_{\alpha\odot(x\odot y)}=R_{\alpha\beta}+R_{\alpha\odot\beta}.\]On the other hand, if $y\in\dot{\N}_0$, then\[R_{\alpha}R_{\beta}=R_{\alpha\beta}+R_{(\alpha x)\odot y}+R_{(\alpha\odot x)y}-R_{\alpha(x\odot y)}=R_{\alpha\beta}+R_{\alpha\odot\beta}.\]Finally, consider $x\in\dot{\N}_0$. As $\{\rg(\alpha),x\}\not\subset\dot{\N}_0$, then $\rg(\alpha)\in\N$. If $y\in\dot{\N}_0$, then\[R_{\alpha}R_{\beta}=(R_{\alpha}R_x)R_y=(R_{\alpha x}+R_{\alpha\odot x})R_y=R_{\alpha\beta}+R_{\alpha\odot\beta},\]because $(\alpha\odot x)y=\alpha\odot(xy)$. If $y\in\N$, then\[\begin{array}{rcl}
R_{\alpha}R_{\beta}&=&(R_{\alpha}R_x)R_y-R_{\alpha}R_{x\odot y},\\
&=&(R_{\alpha x}+R_{\alpha\odot x})R_y-R_{\alpha(x\odot y)}-R_{\alpha\odot(x\odot y)},\\
&=&R_{\alpha\beta}+R_{(\alpha x)\odot y}+R_{(\alpha\odot x)y}+R_{(\alpha\odot x)\odot y}-R_{\alpha(x\odot y)}-R_{\alpha\odot(x\odot y)},\\
&=&R_{\alpha\beta}+R_{\alpha\odot\beta}.
\end{array}\]This concludes the proof.
\end{proof}


\begin{proof}[Proof of Theorem \ref{018}]
Due to Lemma \ref{014} and Lemma \ref{017}, the result is true when $\ell(\beta)\leq2$. So, we will proceed by induction on the length of $\beta$. Assume that $\ell(\beta)\geq3$, that is, $\beta=\gamma x$ for some dotted composition $\gamma$ satisfying $\ell(\gamma)=\ell(\beta)-1$, where $x=\rg(\beta)$. In particular, we have $\beta_1=\gamma_1$ and $\rg(\gamma)\neq\beta_1$. Moreover, by Lemma \ref{014}, we obtain\[R_{\beta}=\left\{\begin{array}{ll}
R_{\gamma}R_x&\text{if }\rg(\gamma),x\in\dot{\N}_0,\\
R_{\gamma}R_x-R_{\gamma\odot x}&\text{otherwise}.
\end{array}\right.\]We will distinguish some cases. 

First, assume that $\rg(\alpha),\beta_1\in\dot{\N}_0$. If $\rg(\gamma),x\in\dot{\N}_0$, by inductive hypothesis, we obtain\[R_{\alpha}R_{\beta}=R_{\alpha}(R_{\gamma}R_x)=(R_{\alpha}R_{\gamma})R_x=R_{\alpha\gamma}R_x=R_{\alpha\beta}.\]On the other hand, if $\{\rg(\gamma),x\}\not\subset\dot{\N}_0$, inductive hypothesis implies that\[\begin{array}{rcl}
R_{\alpha}R_{\beta}&=&R_{\alpha}(R_{\gamma}R_x-R_{\gamma\odot x}),\\
&=&(R_{\alpha}R_{\gamma})R_x-R_{\alpha}R_{\gamma\odot x},\\
&=&R_{\alpha\gamma}R_x-R_{\alpha(\gamma\odot x)},\\
&=&R_{\alpha\beta}+R_{(\alpha\gamma)\odot x}-R_{\alpha(\gamma\odot x)},\\
&=&R_{\alpha\beta},
\end{array}\]because $(\gamma\odot x)_1\in\dot{\N}$ and $(\alpha\gamma)\odot x=\alpha(\gamma\odot x)$.

Now, assume that $\{\rg(\alpha),\beta_1\}\not\subset\dot{\N}_0$. If $\rg(\gamma),x\in\dot{\N}_0$, by inductive hypothesis, we obtain\[\begin{array}{rcl}
R_{\alpha}R_{\beta}&=&(R_{\alpha}R_{\gamma})R_x,\\
&=&(R_{\alpha\gamma}+R_{\alpha\odot\gamma})R_x,\\
&=&R_{\alpha\gamma}R_x+R_{(\alpha\odot\gamma)}R_x,\\
&=&R_{\alpha\gamma x}+R_{(\alpha\odot\gamma)x},\\
&=&R_{\alpha\beta}+R_{\alpha\odot\beta},
\end{array}\]because $\rg(\alpha\gamma)=\rg(\gamma),\rg(\alpha\odot\gamma)\in\dot{\N}_0$ and $(\alpha\odot\gamma)x=\alpha\odot(\gamma x)=\alpha\odot\beta$. On the other hand, if $\{\rg(\gamma),x\}\not\subset\dot{N}_0$, then\[R_{\alpha}R_{\beta}=R_{\alpha}(R_{\gamma}R_x-R_{\gamma\odot x})=(R_{\alpha}R_{\gamma})R_x-R_{\alpha}R_{\gamma\odot x}.\]Since $(\gamma\odot x)_1=\beta_1\neq x$ and $\rg(\alpha\odot\gamma)=\rg(\gamma)$, inductive hypothesis implies that\[\begin{array}{rcl}
R_{\alpha}R_{\beta}&=&(R_{\alpha\gamma}+R_{\alpha\odot\gamma})R_x-R_{\alpha(\gamma\odot x)}-R_{\alpha\odot(\gamma\odot x)},\\
&=&R_{\alpha\gamma}R_x+R_{\alpha\odot\gamma}R_x-R_{\alpha(\gamma\odot x)}-R_{\alpha\odot(\gamma\odot x)},\\
&=&R_{\alpha\beta}+R_{(\alpha\gamma)\odot x}+R_{(\alpha\odot\gamma)x}+R_{(\alpha\odot\gamma)\odot x}-R_{\alpha(\gamma\odot x)}-R_{\alpha\odot(\gamma\odot x)},\\
&=&R_{\alpha\beta}+R_{\alpha\odot\beta},
\end{array}\]because $\alpha\odot(\gamma\odot x)=(\alpha\odot\gamma)x=\alpha\odot(\gamma x)=\alpha\odot\beta$ and $\alpha(\gamma\odot x)=(\alpha\gamma)\odot x$. This concludes the proof.
\end{proof}

Now, by using the previous results, we write noncommutative Ribbon Schur function in superspace in terms of the noncommutative homogeneous functions in superspace.

\begin{pro}\label{047}
Let $\alpha$ be a dotted composition. Then\[R_{\alpha}=\sum_{\alpha\preceq\beta}(-1)^{\ell(\alpha)-\ell(\beta)}H_{\beta}\]
\end{pro}
\begin{proof}
We proceed by induction on the length of $\alpha$. If $\ell(\alpha)=1$ the result is obvious. If $\ell(\alpha)>1$, assume that the result is true for dotted compositions of smaller length. Let $\alpha'$ such that $\alpha=\alpha'x$ with $x=\rg(\alpha)$. We distinguish two cases.

If $\rg(\alpha'),x\in\dot{\N}_0$, Lemma \ref{014} implies that $R_{\alpha}=R_{\alpha'}H_x$. So, by inductive hypothesis\[R_{\alpha}=\sum_{\alpha'\preceq\gamma}(-1)^{\ell(\alpha')-\ell(\gamma)}H_{\gamma}H_x=\sum_{\alpha'\preceq\gamma}(-1)^{\ell(\alpha')+1-\ell(\gamma x)}H_{\gamma x}.\]Since $\rg(\alpha'),x\in\dot{\N}_0$, then $\cv(\alpha)=\{\beta\mid\alpha\preceq\beta\}=\{\gamma x\mid \alpha'\preceq\gamma\}=\cv(\alpha' x)$, hence\[R_{\alpha}=\sum_{\alpha\preceq\beta}(-1)^{\ell(\alpha)-\ell(\beta)}H_{\beta}.\]

Now, if $\{\rg(\alpha'),x\}\not\subset\dot{\N}_0$, Lemma \ref{014} implies that $R_{\alpha}=R_{\alpha'}H_x-R_{\alpha'\odot x}$. So, by inductive hypothesis, we have\[\begin{array}{rcl}
R_{\alpha}&=&{\displaystyle\sum_{\alpha'\preceq\gamma}(-1)^{\ell(\alpha')-\ell(\gamma)}H_{\gamma}H_x-\sum_{\alpha'\odot x\preceq\gamma}(-1)^{\ell(\alpha'\odot x)-\ell(\gamma)}H_{\gamma}},\\[3mm]
&=&{\displaystyle\sum_{\alpha'\preceq\gamma}(-1)^{\ell(\alpha')+1-\ell(\gamma x)}H_{\gamma x}+\sum_{\alpha'\odot x\preceq\gamma}(-1)^{\ell(\alpha'\odot x)+1-\ell(\gamma)}H_{\gamma}},\\[3mm]
&=&{\displaystyle\sum_{\alpha'\preceq\gamma}(-1)^{\ell(\alpha)-\ell(\gamma x)}H_{\gamma x}+\sum_{\alpha'\odot x\preceq\gamma}(-1)^{\ell(\alpha)-\ell(\gamma)}H_{\gamma}},\\[3mm]
&=&{\displaystyle\sum_{\alpha\preceq\beta}(-1)^{\ell(\alpha)-\ell(\beta)}H_{\beta}}.
\end{array}\]
\end{proof}

In what follows of this section, we give explicit formulas to write other noncommutative functions in superspace, defined in previous sections, in terms of the noncommutative Ribbon Schur functions in superspace.

\begin{pro}\label{033}
$\Psi_0=R_{\dot{0}}$, and for every $n\geq1$, we have $\Psi_n=(-1)^nR_{(1^n,\dot{0})}$.
\end{pro}
\begin{proof}
We proceed by induction on $n$. Note that $\Psi_0=\tH_0=R_{\dot{0}}$, and that\[\Psi_1=\tH_1-H_1\Psi_0=R_{\dot{1}}-R_1R_{\dot{0}}=R_{\dot{1}}-R_{(1,\dot{0})}-R_{\dot{1}}=-R_{(1,\dot{0})}.\]Assume the result is true for all $k<n$. So, by definition and inductive hypothesis, we obtain the following\[\Psi_n=\tH_n-\sum_{k=0}^{n-1}H_{n-k}\Psi_k=R_{\dot{n}}-R_nR_{\dot{0}}+R_{n-1}R_{(1,\dot{0})}-\sum_{k=2}^{n-1}R_{n-k}(-1)^kR_{(1^k,\dot{0})}.\]Now, Theorem \ref{018} implies that $R_{n-k}R_{(1^k,\dot{0})}=R_{(n-k,1^k,\dot{0})}+R_{(n-(k-1),1^{k-1},\dot{0})}$, hence\[\begin{array}{rcl}
\Psi_n&=&{\displaystyle R_{(n-1,1,\dot{0})}-\sum_{k=2}^{n-1}(-1)^kR_{(n-k,1^k,\dot{0})}-\sum_{k=2}^{n-1}(-1)^kR_{(n-(k-1),1^{k-1},\dot{0})}},\\
&=&{\displaystyle R_{(n-1,1,\dot{0})}-\sum_{k=2}^{n-1}(-1)^kR_{(n-k,1^k,\dot{0})}+\sum_{k=2}^{n-1}(-1)^{k-1}R_{(n-(k-1),1^{k-1},\dot{0})}},\\
&=&{\displaystyle-\sum_{k=1}^{n-2}(-1)^kR_{(n-k,1^k,\dot{0})}+(-1)^nR_{(1^n,\dot{0})}+\sum_{k=1}^{n-2}(-1)^kR_{(n-k,1^k,\dot{0})}}.
\end{array}\]Therefore $\Psi_n=(-1)^nR_{(1^n,\dot{0})}$.
\end{proof}

Note that $1^n$ is the minimal element in the lattice of compositions of $n$ with the usual order. This implies that\begin{equation}\label{035}R_{(1^n)}=\sum_{1^n\preceq\beta}(-1)^{n-\ell(\beta)}H_{\beta}=(-1)^n\sum_{1^n\preceq\beta}(-1)^{\ell(\beta)}H_{\beta}=(-1)^n\sum_{\beta\preceq(n)}(-1)^{\ell(\beta)}H_{\beta}=(-1)^nS(H_n)=S_n.\end{equation}

In the following proposition, we write $\tS_n$ in terms of noncommutative Ribbon Schur functions in superspace.

\begin{pro}
We have $\tS_n={\displaystyle\sum_{k=0}^nR_{(1^k,\dot{0},1^{n-k})}+\sum_{k=1}^nR_{(1^{k-1},\dot{1},1^{n-k})}}$, for all $n\geq1$.
\end{pro}
\begin{proof}
By Proposition \ref{011}, Proposition \ref{033} and \ref{035}, we have\[\tS_n=(-1)^n\Psi_n+\sum_{k=0}^{n-1}(-1)^k\Psi_kS_{n-k}=R_{(1^n,\dot{0})}+\sum_{k=0}^{n-1}R_{(1^k,\dot{0})}R_{(1^{n-k})}.\]Now, by using Theorem \ref{018}, we obtain the result\[\tS_n=R_{(1^n,\dot{0})}+\sum_{k=0}^{n-1}\left[R_{(1^k,\dot{0},1^{n-k})}+R_{(1^k,\dot{1},1^{n-k-1})}\right]=\sum_{k=0}^nR_{(1^k,\dot{0},1^{n-k})}+\sum_{k=1}^nR_{(1^{k-1},\dot{1},1^{n-k})}.\]
\end{proof}

Recall, for the classic case \cite[Corollary 3.14]{GKLLRT95}, we have\begin{equation}\label{032}P_n={\displaystyle\sum_{k=0}^{n-1}(-1)^kR_{(1^k,n-k)}},\qquad n\geq1.\end{equation}

To finish, we give a explicit formula of $\tP_n$ in terms of noncommutative Ribbon Schur functions in superspace.
\begin{pro}\label{034}
For every $n\geq1$, we have
\[(n+1)\tP_n=-R_{(\dot{0},n)}+(-1)^nR_{(1^n,\dot{0})}+\sum_{k=1}^{n-1}(-1)^{k+1}\left(\sum_{i=1}^kR_{(1^{i-1},\dot{1},1^{k-i},n-k)}+\sum_{i=0}^kR_{(1^i,\dot{0},1^{k-i},n-k)}\right).\]
\end{pro}
\begin{proof}
The result follows by applying inductively Proposition \ref{006}, \ref{032} and Theorem \ref{018}.
\end{proof}

\section{Related structures}\label{030}

\subsection{$\sNSym$ as a Hopf algebra of trees}\label{029}

In this subsection, we give a realization of $\sNSym$ as a Hopf algebra of planar rooted trees.

Hopf algebras generated by trees have been widely studied and there is an extensive literature on it. See for example \cite{Ho03,PreFo09,ArMa21}. The \emph{Connes--Kreimer Hopf algebra}, generated by rooted trees, was introduced in \cite{CoKr98} and its noncommutative version $H_{PR}$ for planar rooted trees was given simultaneously in \cite{Fo02,Ho03}. The coalgebra structure of these Hopf algebras can be described in terms of cuts of trees.

It was shown in \cite{PreFo09,Hof09} that $\NSym$ can be released as a Hopf subalgebra of $H_{PR}$. Here, we extend this description for $\sNSym$, by identifying its generators with certain type of planar rooted trees. Cf. \cite{Dol21}.

For a planar rooted tree $t$, we define the \emph{degree} of it, denoted by $\deg(t)$, as the number of its non-root nodes. Recall that a \emph{ladder tree} is a planar rooted tree with only one branch. For $n\geq0$, we will denote by $t_n$ the unique ladder tree of degree $n$, and we will identify $H_n$ with this tree. Respectively, the generator $\tH_n$ is identified with the planar rooted tree $t_{\dot{n}}$ of degree $n+1$ obtained by gluing the roots of a coloured ladder tree of degree one with $t_n$. The coloured node of $t_{\dot{n}}$ represents the fermionic degree of $\tH_n$. For instance, for $n\leq 2$, we have:\begin{figure}[H]\figsix\end{figure}Given a dotted composition $\alpha$ with $k=\ell(\alpha)$, we identify $H_{\alpha}=H_{\alpha_1}\cdots H_{\alpha_k}$ with the forest $t_{\alpha}=t_{\alpha_1}\cdots t_{\alpha_k}$, and $1$ is identified with $t_0$. For instance, for $\alpha=(\dot{2},\dot{0},2,\dot{3})$, we have:\begin{figure}[H]\figsev\end{figure}Hence, the product is obtained by concatenating forests with the assumption that $t_{\alpha}t_0=t_0t_{\alpha}=t_{\alpha}$ for all dotted composition $\alpha$.

To the describe the coproduct, we consider \emph{admissible cuts} on trees. A \emph{cut} on a planar rooted tree $t$ is any subset of edges of it. A cut is called \emph{admissible} if each branch of the tree contains at most one edge of it. The set of all admissible cuts of a tree $t$ is denoted by $\adm(t)$. Note that the empty cut of $t$ is admissible.

Given a planar rooted tree $t$ and an admissible cut $c$ of it, we call \emph{components} of $t$ respect to $c$, the subtrees of $t$ obtained by removing the edges of $c$ from $t$. We denote by $R^c(t)$ the component containing the root of $t$. On the other hand, by adapting the classic case, we denote by $P^c(t)$ the planar tree obtained by removing the non-root nodes of $R^c(t)$ from $t$ and then contracting its edges. For instance,\figtwe
It is easy to see that via the identification $H_m\mapsto t_m$, with $m\in\N\cup\dot{\N}_0$, the coproduct of $\sNSym$ can be described by means the following formula\begin{equation}\label{049}\Delta(t_m)=\sum_{c\in\adm(t)}P^c(t)\otimes R^c(t).\end{equation}

The notions described above can be extended to the forests $t_{\alpha}=t_{\alpha_1}\cdots t_{\alpha_k}$, with $\alpha$ a dotted composition of length $k$. Indeed, an \emph{admissible cut} $c$ of $t_{\alpha}$ is a tuple $c=(c_1,\ldots,c_k)$, where each $c_i$ is an admissible cut of $t_{\alpha_i}$, possibly empty. Thus, $P^c(t_{\alpha})$ and $R^c(t_{\alpha})$ are given by the following forests:\[P^c(t_{\alpha})=P^{c_1}(t_{\alpha_1})\cdots P^{c_k}(t_{\alpha_k}),\qquad R^c(t_{\alpha})=R^{c_1}(t_{\alpha_1})\cdots R^{c_k}(t_{\alpha_k}).\]

We will see in Proposition \ref{048} that the addends of $\Delta(t_{\alpha})$ can be described by $P^c(t_{\alpha})\otimes R^c(t_{\alpha})$ up a sign induced by the fermionic components of $\alpha$. 

Let $\alpha$ be a dotted composition with $j=\df(\alpha)$, and let $c$ be a cut of $t_{\alpha}$. If $j\geq1$, we denote by $\sigma_c$ to be the unique permutation in $\Sym_j$ given by the reordering of the coloured nodes of $t_{\alpha}$ in $P^c(t_{\alpha})\otimes R^c(t_{\alpha})$. We define the sign of $c$ as follows:\[\sgn(c)=\left\{\begin{array}{ll}1&\text{if }j=0,\\\sgn(\sigma_c)&\text{otherwise}.\end{array}\right.\]For instance, if $\alpha=(\dot{2},\dot{0},2,\dot{3},\dot{1})$, we have:\figthi where $\sigma_c=(2,3,4,1)$ and $\sgn(c)=\sgn(\sigma_c)=(-1)^3=-1$.

\begin{pro}\label{048}
For a dotted composition $\alpha$, we have\[\Delta(t_{\alpha})=\sum_{c\in\adm(t_{\alpha})}\sgn(c)P^c(t_{\alpha})\otimes R^c(t_{\alpha}).\]
\end{pro}
\begin{proof}
Let $j=\df(\alpha)$. We proceed by induction on $k:=\ell(\alpha)$. If $k=1$, the result follows from \ref{049}. If $k>1$, assume the claim is true for smaller values. We will show that the addends on both sides of the equation coincide. Notice that $\Delta(t_{\alpha})=\Delta(t_{\alpha'})\Delta(t_{\alpha_k})$, where $\alpha'=(\alpha_1,\ldots,\alpha_{k-1})$. Now, by inductive hypothesis, the addends of $\Delta(t_{\alpha'})$ can be written as $\sgn(c')P^{c'}(t_{\alpha'})\otimes R^{c'}(t_{\alpha'})$, where $c'=(c_1,\ldots,c_{k-1})$ is an admissible cut of $t_{\alpha'}$. Respectively, the addends of $\Delta(t_{\alpha_k})$ can be written as $P^{c_k}(t_{\alpha_k})\otimes R ^{c_k}(t_{\alpha_k})$, where $c_k$ is an admissible cut of $t_{\alpha_k}$. Hence, an addend of $\Delta(t_{\alpha})$ has the following form:\[[\sgn(c')P^{c'}(t_{\alpha'})\otimes R^{c'}(t_{\alpha'})]\cdot[P^{c_k}(t_{\alpha_k})\otimes R ^{c_k}(t_{\alpha_k})]=\sgn(c')(-1)^{ab}P^{c'}(t_{\alpha'})P^{c_k}(t_{\alpha_k})\otimes R^{c'}(t_{\alpha'})R^{c_k}(t_{\alpha_k}),\]where $a=\df(R^{c'}(t_{\alpha'}))$ and $b=\df(P^{c_k}(t_{\alpha_k}))$. Observe that $c=(c_1,\ldots,c_{k-1},c_k)$ is an admissible cut of $t_{\alpha}$, satisfying\begin{equation}\label{050}P^c(t_{\alpha})\otimes R^c(t_{\alpha})=P^{c'}(t_{\alpha'})P^{c_k}(t_{\alpha_k})\otimes R^{c'}(t_{\alpha'})R^{c_k}(t_{\alpha_k}).\end{equation}Now, we will show that $\sgn(c)=\sgn(c')(-1)^{ab}$. If $j=0$ it is obvious. For $j\geq1$, we distinguish three cases. 

If $\alpha_k\in\N$, then $\sigma_c=\sigma_{c'}$ and $b=0$.

If $\alpha_k\in\dot{\N}_0$ and $b=0$, then $\sgn(c')(-1)^{ab}=\sgn(\sigma_{c'})$. Since the coloured node of $\Delta(t_{\alpha_k})$ belongs to $R^{c_k}(t_{\alpha_k})$, then $\sigma_c=(\sigma_{c'}(1),\ldots,\sigma_{c'}(j-1),j)$. Thus, the number of inversions of $\sigma_{c'}$ and $\sigma_c$ coincide, and so $\sgn(\sigma_{c'})=\sgn(c)$.

If $\alpha_k\in\dot{\N}_0$ and $b=1$, then $\sgn(c')(-1)^{ab}=\sgn(\sigma_{c'})(-1)^a$. Since $\sigma_{c'}$ is determined by the order of coloured nodes of $P^{c'}(t_{\alpha'})\otimes R^{c'}(t_{\alpha'})$, there is $i\in[j-1]$ such that the nodes positioned in $\sigma_{c'}(1),\ldots,\sigma_{c'}(i)$ belong to $P^{c'}(t_{\alpha'})$ and the nodes positioned in $\sigma_{c'}(i+1),\ldots,\sigma_{c'}(j-1)$ belong to $R^{c'}(t_{\alpha'})$. This together with \ref{050} imply that $a=j-1-i$ and $\sigma_c=(\sigma_{c'}(1),\ldots,\sigma_{c'}(i),j,\sigma_{c'}(i+1),\ldots,\sigma_{c'}(j-1))$. Thus, the number of inversions of $\sigma_c$ is the one of $\sigma_{c'}$ plus $j-1-i$. So, $\sgn(\sigma_c)=\sgn(\sigma_{c'})(-1)^{j-1-i}=\sgn(\sigma_{c'})(-1)^a$.

Similarly, we prove that for each admissible cut $c$ of $t_{\alpha}$, $\sgn(c)P^c(t_{\alpha})\otimes R^c(t_{\alpha})$ can be represented as an addend of $\Delta(t_{\alpha})$. This concludes the proof.
\end{proof}

\subsection{Fundamental quasisymmetric functions in superspace}\label{028}

The dual structure of $\sNSym$ is the Hopf algebra of quasisymmetric functions in superspace $\sQSym$. This relation is determined by a pairing $\langle\cdot,\cdot\rangle:\sQSym\otimes\sNSym\to\Q$, which satisfies $\langle M_{\alpha},H_{\beta}\rangle=\delta_{\alpha\beta}$.

Now, we introduce the set of \emph{fundamental quasisymmetric functions in superspace} $\{L_{\alpha}\}$ as the basis of $\sQSym$ obtained by dualizing the noncommutative Ribbon Schur functions in superspace $\{R_{\beta}\}$, defined in Section \ref{027}, that is $\langle L_{\alpha},R_{\beta}\rangle=\delta_{\alpha\beta}$. In the following theorem, we show that this basis coincides with the set of fundamental quasisymmetric functions introduced in \cite{FiLaPi19}, with respect to the partial order $\preceq$ defined in Section \ref{021}. Moreover, we write $M_{\alpha}$ in terms of $L_{\beta}$ and provide a formula for the coproduct of fundamental quasisymmetric functions in superspace.

\begin{thm}
Let $\alpha$ be a dotted composition. Then\[L_{\alpha}=\sum_{\beta\preceq\alpha}M_{\beta},\qquad M_{\alpha}=\sum_{\beta\preceq\alpha}(-1)^{\ell(\alpha)-\ell(\beta)}L_{\beta},\qquad\Delta(L_{\alpha})=\sum_{\beta\gamma=\alpha\text{ or }\beta\odot\gamma=\beta}L_{\beta}\otimes L_{\gamma}.\]
\end{thm}
\begin{proof}
As $\{M_{\gamma}\}$ forms a basis of $\sQSym$, we can write $L_{\alpha}$ as a linear combination ${\displaystyle L_{\alpha}=\sum_{\gamma}c_{\gamma}M_{\gamma}}$. 

For a dotted composition $\beta$, we have ${\displaystyle\langle L_{\alpha},H_{\beta}\rangle=\sum_{\gamma}c_{\gamma}\langle M_{\gamma},H_{\beta}\rangle.}$

Since $\langle M_{\gamma},H_{\beta}\rangle=0$ if $\gamma\neq\beta$ and $\langle M_{\gamma},H_{\beta}\rangle=0$ if $\gamma=\beta$, then $c_{\beta}=\langle L_{\alpha},H_{\beta}\rangle$. This implies that\[L_{\alpha}=\sum_{\beta}\langle L_{\alpha},H_{\beta}\rangle M_{\beta}.\]Now, as $H_{\beta}={\displaystyle\sum_{\beta\preceq\gamma}R_{\gamma}}$, then ${\displaystyle\langle L_{\alpha},H_{\beta}\rangle=\sum_{\beta\preceq\gamma}\langle L_{\alpha},R_{\gamma}\rangle=\left\{\begin{array}{ll}0&\text{if }\alpha\prec\beta,\\1&\text{if }\alpha\succeq\beta.
\end{array}\right.}$ Thus, ${\displaystyle L_{\alpha}=\sum_{\beta\preceq\alpha}M_{\beta}}$.

Similarly, as $\{L_{\alpha}\}$ is dual to the basis $\{R_{\alpha}\}$, then it is also a basis of $\sQSym$. Thus,\[M_{\alpha}=\sum_{\beta}\langle M_{\alpha},R_{\beta}\rangle L_{\beta}=\sum_{\beta}\langle M_{\alpha},\sum_{\beta\preceq\gamma}(-1)^{\ell(\beta)-\ell(\gamma)}H_{\gamma}\rangle L_{\beta}=
\sum_{\beta}\left[\sum_{\beta\preceq\gamma}\langle M_{\alpha},(-1)^{\ell(\beta)-\ell(\gamma)}H_{\gamma}\rangle\right]L_{\beta}.\]

Now, ${\displaystyle\sum_{\beta\preceq\gamma}\langle M_{\alpha},(-1)^{\ell(\beta)-\ell(\gamma)}H_{\gamma}\rangle}$ is $0$ if $\alpha\prec\beta$ and it is $(-1)^{\ell(\beta)-\ell(\alpha)}$ otherwise. Hence,\[M_{\alpha}=\sum_{\beta\preceq\alpha}(-1)^{\ell(\beta)-\ell(\alpha)}L_{\beta}.\]The last assertion follows from Theorem \ref{018} and the duality.
\end{proof}

For a dotted composition $\alpha$, the coproduct of $L_{\alpha}$ can be obtained by considering all possible horizontally ($\alpha=\beta\gamma$) and vertically ($\alpha=\beta\odot\gamma$) splitting of the Ribbon diagram of $\alpha$. For instance, for $\alpha=(1,\dot{3},2)$, we have\figftn Note that, to obtain all possible vertically splitting of $\alpha$ we need to consider both the left and the right diagram of it. This description extends the one for classic Ribbon diagrams, see \cite[Proposition 5.2.15]{GrRe14}. Thus,\figfif

We also can describe the coproduct of $L_{\alpha}$, with $\alpha$ a length $k$ dotted composition, by identifying $L_{\alpha}$ with the forest $t_{\alpha}$. Thus,\[\Delta(t_{\alpha})=1\otimes t_{\alpha}+\sum_{i=1}^k\sum_{c\neq\emptyset}t_{(\alpha_1,\ldots,\alpha_{i-1})}P^c(t_{\alpha_i})\otimes R^c(t_{\alpha_i})t_{(\alpha_{i+1},\ldots,\alpha_k)}.\]For instance, for $\alpha=(1,\dot{3},2)$, we have\figsvt

\subsection{Symmetric functions in superspace}\label{031}

In this subsection, we present some results on symmetric functions in superspace, which are obtained by the projection $\pi$ from $\sNSym$ to $\sSym$ defined in \cite{FiLaPi19}. Additionally, we obtain a new basis of $\sSym$ formed by a new class of functions that we will call \emph{Ribbon Schur functions in superspace}, which extends the classic Ribbon Schur functions in $\SSym$.

Recall that $\pi:\sNSym\to\sSym$ is determined by $\pi(H_n):=h_n$ and $\pi(\tH_n):=\th_n$. In particular, this morphism extends the classic morphism from $\NSym$ to $\SSym$.

The following proposition characterizes the action of $\pi$ on the families of noncommutative functions in superspace.

\begin{pro}\label{039}
For every $n\geq0$, we have $\pi(\tS_n)=\te_n$ and $\pi(\tP_n)=\pi(\Psi_n)=\tp_n$.
\end{pro}
\begin{proof}
First, as $\pi$ is a Hopf algebra morphism, we have $\pi\circ S=S\circ\pi$, where $S$ denotes the antipodes of $\sNSym$ and $\sSym$ respectively. Thus, we obtain $\pi(S(\tH_n))=\pi((-1)^{n+1}\tS_n)=(-1)^{n+1}\pi(\tS_n)$, and $S(\pi(\tH_n))=S(\th_n)=(-1)^{n+1}\te_n$ \cite[Corollary 4.5]{FiLaPi19}. So, $(-1)^{n+1}\pi(\tS_n)=(-1)^{n+1}\te_n$. Hence, $\pi(\tS_n)=\te_n$. 

For the second part, we proceed by induction on $n$. Since $\tp_0=\th_0$ and $\tP_0=\tH_0$, then $\pi(\tP_0)=\tp_0$. Now, let $n\geq1$, and assume the result is true for $0\leq k<n$. By \cite[Lemma 26]{DeLaMa06}, we have\begin{equation}\label{036}
(n+1)\tp_n=(n+1)\th_n-p_n\th_0-\sum_{k=0}^{n-1}\left(p_k\th_{n-k}+(k+1)\tp_kh_{n-k}\right).\end{equation}
By applying $\pi$ on both sides of \ref{037}, inductive hypothesis implies that $\pi(\tP_n)=\tp_n$. Similarly, by applying $\pi$ on \ref{038}, we obtain $\pi(\Psi_n)=\pi(\tP_n)$, because the Lie brackets become zero due to the commutativity of the product in $\sSym$ whenever one of the elements has null fermionic degree.
\end{proof}

\begin{pro}
For $n\geq1$, we have ${\displaystyle\tp_n=\th_n-\sum_{k=0}^{n-1}h_{n-k}\tp_k}$ and ${\displaystyle\te_n=\sum_{k=0}^n(-1)^{n-k}\tp_{n-k}e_k}$.
\end{pro}
\begin{proof}
It is a consequence of Proposition \ref{039}, \ref{019} and Proposition \ref{009}.
\end{proof}

Now, we introduce the \emph{Ribbon Schur functions in superspace}. In the classic case, Ribbon Schur functions can be regarded as a special case of \emph{Skew Schur functions} $s_{\lambda/\mu}$, which are indexed by the so-called \emph{skew partitions} $\lambda/\mu$, where $\lambda,\mu$ are partitions such that the Young diagram of $\mu$ is contained in the one of $\lambda$. The Young diagram of $\lambda/\mu$ is obtained by removing the boxes of the diagram of $\mu$ from the one of $\lambda$. This diagram is called a \emph{Ribbon diagram} if it is connected and contains no a $2\times2$ block of boxes. In this case, we can identity $\lambda/\mu$ with a composition $\alpha$. Thus, the \emph{Ribbon Schur function} $r_{\alpha}$ is defined as $s_{\lambda/\mu}$. See \cite[Section 7.15]{St99} for details.

In superspace, Schur functions were defined by means a specialization of the parameters of the so-called \emph{Macdonald polynomials in superspace}. On the other hand, \emph{Skew Schur functions in superspace} were defined in relation with a generalization of the Littlewood--Richardson coefficients, see \cite{JoLa17} for details.

Here, for a dotted composition $\alpha$, we define the \emph{Ribbon Schur functions in superspace} $r_{\alpha}$ by projecting the noncommutative Ribbon Schur functions in superspace $R_{\alpha}$ on $\sSym$, that is, $r_{\alpha}=\pi(R_{\alpha})$.

In what follows of this subsection we will identify a superpartition with the unique dotted composition obtained by dotting its fermionic part.

Recall that if $\alpha$ is a usual composition, $\tilde{\alpha}$ denotes the partition obtained by sorting its component in nonincreasing order. Similarly, for a dotted composition $\alpha$, we will denote by $\tilde{\alpha}$ the tuple obtained by sorting in nonincreasing order, both the dotted components and the undotted components. Note that, since the fermionic part of a superpartition must be strictly decreasing, $\tilde{\alpha}$ is a superpartition only if the dotted components of $\alpha$ are all different. For instance, if $\alpha=(\dot{1},2,3,\dot{2},3,1,\dot{4})$, then $\tilde{\alpha}=(\dot{4},\dot{2},\dot{1},3,3,2,1)$, that is\figstn

For a dotted composition $\alpha$ with $k=\ell(\alpha)$, we define $h_{\alpha}=h_{\alpha_1}\cdots h_{\alpha_k}$. Note that, as $h_m^2=0$ for all $m\in\dot{\N}_0$, then $h_{\alpha}=0$ whenever $\alpha$ has repeated dotted components. Further, if $\tilde{\alpha}$ is a superpartition, we have $h_{\alpha}=(-1)^{\sigma(\alpha)}h_{\tilde{\alpha}}$, where $\sigma(\alpha)$ is the number of inversions of the permutation of the dotted components of $\alpha$, obtained when computing $\tilde{\alpha}$. For instance, $h_{(\dot{1},2,3,\dot{2},3,1,\dot{4})}=-h_{(\dot{4},\dot{2},\dot{1},3,3,2,1)}$.

Below, we obtain an expansion of a Ribbon Schur function in superspace in terms of complete homogeneous functions in superspace, which generalizes the well-known formula for classic Ribbon Schur functions \cite{MacMahon12}.

\begin{pro}\label{053}
For a dotted composition $\alpha$, we have\[r_{\alpha}=\sum_{\alpha\preceq\beta}(-1)^{\ell(\alpha)-\ell(\beta)+\sigma(\beta)}h_{\tilde{\beta}}.\]Conversely, for a superpartition $\Lambda$, we have ${\displaystyle h_{\Lambda}=\sum_{\Lambda\preceq\beta}r_{\beta}}$.
\end{pro}
\begin{proof}
This is a consequence of Proposition \ref{047}, the definition of $r_{\alpha}$ and the fact that $\pi(H_{\beta})=0$ whenever $\beta$ has repeated dotted components. 
\end{proof}

For instance, for $\alpha=(\dot{0},1,\dot{2},1)$, we have\[r_{\alpha}=-h_{(\dot{2},\dot{0},1,1)}+2h_{(\dot{3},\dot{0},1)}+h_{(\dot{2},\dot{1},1)}-h_{(\dot{3},\dot{1})}-h_{(\dot{4},\dot{0})}.\]

It follows from Theorem \ref{018} that, for dotted compositions $\alpha,\beta$, the product of $r_{\alpha}$ with $r_{\beta}$ is given as follows:\begin{equation}\label{051}r_{\alpha}r_{\beta}=\left\{\begin{array}{ll}
r_{\alpha\beta}&\text{if }\rg(\alpha),\beta_1\in\dot{\N}_0,\\
r_{\alpha\beta}+r_{\alpha\odot\beta}&\text{otherwise}.
\end{array}\right.\end{equation}

\begin{pro}\label{052}
For every $n\geq0$, we have ${\displaystyle r_{(1^n,\dot{0})}=\sum_{k=0}^n(-1)^{k}r_{(1^{n-k})}r_{(\dot{k})}}$.
\end{pro}
\begin{proof}
We proceed by induction on $n$. The result is obvious if $n=0$. Now, assume the claim is true for values less than or equal to $n$. Thus,\[r_{(1^n,\dot{0})}=\sum_{k=0}^n(-1)^kr_{(1^{n-k})}r_{(\dot{k})}.\]By multiplying by $r_{(1)}$ on both sides of the equation above, with respect to the product in \ref{051}, and by applying the inductive hypothesis, we obtain\[\begin{array}{rcl}
r_{(1^{n+1},\dot{0})}+r_{(2,1^{n-1},\dot{0})}&=&{\displaystyle\sum_{k=0}^{n-1}(-1)^kr_{(1^{n+1-k})}r_{(\dot{k})}+\sum_{k=0}^{n-1}(-1)^kr_{(2,1^{n-1-k})}r_{(\dot{k})}+(-1)^nr_{(1)}r_{(\dot{n})},}\\[0.6cm]
&=&{\displaystyle\sum_{k=0}^{n-1}(-1^k)r_{(1^{n+1-k})}r_{(\dot{k})}+\sum_{k=0}^{n-1}(-1)^k(r_{(2,1^{n-1-k},\dot{k})}+r_{(2,1^{n-k-2},\dot{(k+1)})})+(-1)^nr_{(\dot{n+1})},}\\[0.6cm]
&=&{\displaystyle\sum_{k=0}^{n-1}(-1^k)r_{(1^{n+1-k})}r_{(\dot{k})}+r_{(2,1^{n-1},\dot{0})}+(-1)^{n-1}r_{\dot{(n+1)}}+(-1)^nr_{(1)}r_{(\dot{n})}.}
\end{array}\]Therefore, ${\displaystyle r_{(1^{n+1},\dot{0})}=\sum_{k=0}^{n+1}r_{(1^{n+1-k})}r_{(\dot{k})}}$.
\end{proof}

\begin{pro}
We have:
\[\tp_n=(-1)^nr_{(1^n,\dot{0})},\qquad r_{(1^n,\dot{0})}=\sum_{k=0}^n(-1)^k e_{n-k}\th_k,\qquad\tp_n=\sum_{k=0}^n(-1)^{n-k}e_{n-k}\th_k.\]
\end{pro}
\begin{proof}
The first equality follows directly by applying $\pi$ in Proposition \ref{033}. The second assertion is a consequence of Proposition \ref{052} and the fact that, in the classic case, $r_{(1^i)}=e_i$, where $e_i$ is the classic elementary symmetric function and $r_{(\dot{k})}=\th_k$. The last result follows directly from the first two equalities.
\end{proof}

Recall that the set of classic Ribbon Schur functions indexed by partitions is a basis of $\SSym$ \cite{BiThoWi06}. We conclude this section by extending this result to superspace.

For dotted compositions $\alpha,\beta$ with $a=\df(\alpha)$ and $b=\df(\beta)$, we set $r_{\alpha}\ast r_{\beta}=0$ if $\rg(\alpha)$ and $\beta_1$ are dotted, and $r_{\alpha}\ast r_{\beta}=r_{\alpha\odot\beta}$ otherwise. Thus, as $r_{\alpha}r_{\beta}=(-1)^{ab}r_{\beta}r_{\alpha}$, the product \ref{051} implies the following:\begin{equation}\label{054}r_{\alpha\beta}=(-1)^{ab}r_{\beta}r_{\alpha}-r_{\alpha}\ast r_{\beta}.\end{equation}

\begin{lem}\label{055}
Let $\alpha$ be a dotted composition such that $\tilde{\alpha}$ is a superpartition and $\df(\alpha)\geq1$. Then\[r_{\alpha}=(-1)^{\sigma(\alpha)}r_{\tilde{\alpha}}+\sum_{\mu}c_{\mu}r_{\mu}\quad\text{where}\quad\ell(\mu)<\ell(\alpha)\quad\text{and}\quad c_{\mu}\in\Z.\]
\end{lem}
\begin{proof}
We proceed by induction on $k:=\ell(\alpha)$. If $k=1$ the result is obvious. If $k>1$, assume the claim is true for smaller values. Let $\alpha_i$ be the maximal dotted component of $\alpha$. Consider $\beta=(\alpha_1,\ldots,\alpha_{i-1})$ with $b=\df(\beta)$ and $\gamma=(\alpha_i,\ldots,\alpha_k)$ with $c=\df(\gamma)$. Then, by \ref{051} and \ref{054}, we have\[\begin{array}{rcl}
r_{\alpha}&=&(-1)^{bc}r_{\gamma}r_{\beta}-r_{\beta}\ast r_{\gamma},\\
&=&(-1)^{bc}r_{(\alpha_i,\ldots,\alpha_k,\alpha_1,\ldots,\alpha_{i-1})}+(-1)^{bc}r_{\gamma}\ast r_{\beta}-r_{\beta}\ast r_{\gamma},\\
&=&(-1)^{bc}r_{\alpha_i}r_{(\alpha_{i+1},\ldots,\alpha_k,\alpha_1,\ldots,\alpha_{i-1})}+(-1)^{bc+1}r_{\alpha_i}\ast r_{(\alpha_{i+1},\ldots,\alpha_k,\alpha_1,\ldots,\alpha_{i-1})}+(-1)^{bc}r_{\gamma}\ast r_{\beta}-r_{\beta}\ast r_{\gamma}.
\end{array}\]Now, by applying the inductive hypothesis on $\omega:=(\alpha_{i+1},\ldots,\alpha_k,\alpha_1,\ldots,\alpha_{i-1})$, we obtain\[r_{\alpha}=(-1)^{bc+\sigma(\omega)}r_{\alpha_i\tilde{\omega}}+\sum_{\mu}c_{\mu}r_{\mu}=(-1)^{\sigma(\alpha)}r_{\tilde{\alpha}}+\sum_{\mu}c_{\mu}r_{\mu},\]where $\ell(\mu)<\ell(\alpha)$.
\end{proof}

\begin{thm}
The set $\{r_{\Lambda}\mid\Lambda\mbox{ is a superpartition}\}$ is a basis of $\sSym$.
\end{thm}
\begin{proof}
Let $\alpha$ be a dotted composition. It is enough to show that $r_{\alpha}$ is a linear combination of elements in $\{r_{\Lambda}\}$. If $\df(\alpha)=0$, the result is true due to the classic case. For $\df(\alpha)\geq1$, we proceed by induction on $k:=\ell(\alpha)$. If $k=1$, the result is obvious. If $k>1$, we assume the claim is true for dotted compositions of smaller length. We will distinguish two cases. 

If $\alpha$ has repeated dotted components, then $h_{\alpha}=0$. So, by Proposition \ref{053}, we have\[r_{\alpha}=\sum_{\alpha\prec\beta}(-1)^{\ell(\alpha)-\ell(\beta)+\sigma(\beta)}h_{\tilde{\beta}}=\sum_{\alpha\prec\beta}(-1)^{\ell(\alpha)-\ell(\beta)+\sigma(\beta)}\sum_{\alpha_{\tilde{\beta}}\preceq\gamma}r_{\gamma}.\]Since $\ell(\gamma)\leq\ell(\beta)<\ell(\alpha)$ for each $\gamma$ above, the result is obtained by applying the inductive hypothesis to $r_{\gamma}$.

Now, assume that $\tilde{\alpha}$ is a superpartition. Lemma \ref{055} implies that\[r_{\alpha}=(-1)^{\sigma(\alpha)}r_{\tilde{\alpha}}+\sum_{\mu}c_{\mu}r_{\mu}\quad\text{where}\quad\ell(\mu)<\ell(\alpha)\quad\text{and}\quad c_{\mu}\in\Z.\]We conclude the proof by applying the inductive hypothesis on each $r_{\mu}$.
\end{proof}

\subsubsection*{Acknowledgements}

This work is part of the research group GEMA Res.180/2019 VRIP--UA. The second author was supported by the Chilean posdoctoral grant ANID--FONDECYT, No. 3200447.



\bibliographystyle{plainurl}
\bibliography{../../../Documents/Projects/LaTeX/bibtex}

\end{document}

%% file: pics.tex
\newcommand{\activate}{1}
\newcommand{\aspdf}{1}

\usepackage{tikz}
\usetikzlibrary{snakes}

\newcommand{\dotcolor}{80}
\newcommand{\nodcolor}{10}
\newcommand{\lincolor}{70}

\newcommand{\vcdraw}[1]{\vcenter{\hbox{#1}}}

\ifnum\activate=0
	\usepackage[
		colordot=black!\dotcolor,
		colorsquare=black!\nodcolor,
		colordraw=black!\lincolor,
		externalize=0,
		sizedot=4, 
		sizesquare=7, 
		widthlines=0.8
	]{../../../Documents/Projects/LaTeX/partitions}
\fi

\newcommand{\gen}[4]{\kern-1em\begin{array}[t]{c}	#1\\[#2]\kern+#3 #4\end{array}\kern-0.55em}

\usetikzlibrary{arrows}
\usetikzlibrary{graphs}

\newcommand{\nodeminimum}{0.25em}
\newcommand{\nodetextwidth}{0.45em}
\newcommand{\nodelevel}{0.23cm}
\newcommand{\nodesibling}{0.4cm}

\tikzset{
	rnode/.style={
  		circle,
  		draw=black!\lincolor,
  		fill=black!\lincolor,
  		inner sep=0pt,
 		minimum height=\nodeminimum,
  		minimum width=\nodeminimum,
	},
	fnode/.style={
  		align=center,
  		circle,
  		draw=black!\dotcolor,
  		inner sep=0pt,
  		fill=black!\dotcolor,
  		font=\tiny,
  		text centered,
  		text width=\nodetextwidth,
	},
	tnode/.style={
  		align=center,
  		circle,
  		draw=black!\lincolor,
  		inner sep=0pt,
  		fill=black!\nodcolor,
  		font=\tiny,
  		text centered,
  		text width=\nodetextwidth,
	},
	trees/.style={
		level/.style={
			level distance=\nodelevel,
			sibling distance=\nodesibling
		},
		grow=up
	},
	child/.style={
		draw=black!\lincolor
	},
	cuts/.style={
		red,
		snake=snake,
		segment amplitude=0.5,
		segment length=1.5,
		line after snake=0
	}
}

\newcommand{\tnode}{child[child]{node[tnode]{}}}

\newcommand{\figone}{
	\ifnum\aspdf=1\[
		\vcdraw{\includegraphics[scale=0.9]{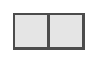}}\cdots
		\vcdraw{\includegraphics[scale=0.9]{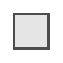}}\kern+4em
		\vcdraw{\includegraphics[scale=0.9]{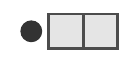}}\cdots
		\vcdraw{\includegraphics[scale=0.9]{pics/001.pdf}}\kern+4em
		\vcdraw{\includegraphics[scale=0.9]{pics/000.pdf}}\cdots
		\vcdraw{\scalebox{-1}[1]{\includegraphics{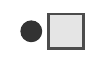}}}
	\]\else
		\tabpart{\snode\&\snode}
		\tabpart{\snode}
		\tabpart{\dnode\&\snode\&\snode}
		\rtabpart{\dnode\&\snode}
	\fi
}

\newcommand{\figtwo}{
	\ifnum\aspdf=1\[
		\vcdraw{\includegraphics[scale=0.9]{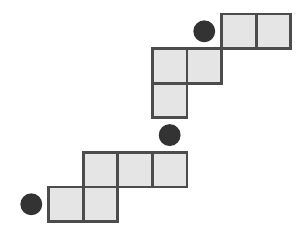}}\qquad
		\vcdraw{\includegraphics[scale=0.9]{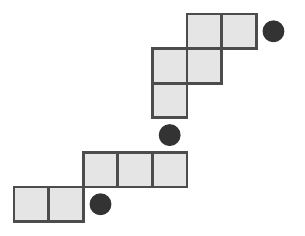}}
	\]\else
		\tabpart{
			\&\&\&\&\&\dnode\&\snode\&\snode\\
			\&\&\&\&\snode\&\snode\\
			\&\&\&\&\snode\\
			\&\&\&\&\dnode\\
			\&\&\snode\&\snode\&\snode\\
			\dnode\&\snode\&\snode
		}
		\tabpart{
			\&\&\&\&\&\snode\&\snode\&\dnode\\
			\&\&\&\&\snode\&\snode\\
			\&\&\&\&\snode\\
			\&\&\&\&\dnode\\
			\&\&\snode\&\snode\&\snode\\
			\snode\&\snode\&\dnode
		}		
	\fi
}

\newcommand{\figthr}{
	\ifnum\aspdf=1
		\includegraphics{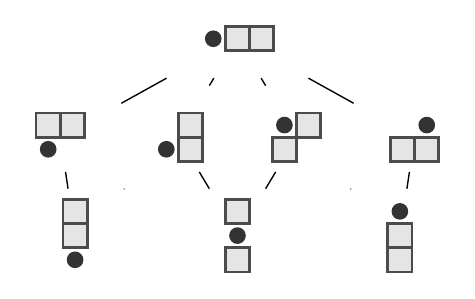}
	\else
		\begin{tikzpicture}
			\node at(0,2)(0){\tabpart{\dnode\&\snode\&\snode}};
			\node at(-1.8,1)(1){\tabpart{\snode\&\snode\\\dnode}};
			\node at(-0.6,1)(2){\tabpart{\&\snode\\\dnode\&\snode}};
			\node at(+0.6,1)(3){\tabpart{\dnode\&\snode\\\snode}};
			\node at(+1.8,1)(4){\tabpart{\&\dnode\\\snode\&\snode}};
			\node at(-1.65,0)(5){\tabpart{\snode\\\snode\\\dnode}};
			\node at(+0.00,0)(6){\tabpart{\snode\\\dnode\\\snode}};
			\node at(+1.65,0)(7){\tabpart{\dnode\\\snode\\\snode}};
			\draw(0)--(1);\draw(0)--(2);\draw(0)--(3);\draw(0)--(4);\draw(1)--(5);\draw(2)--(5);\draw(2)--(6);\draw(3)--(6);\draw(3)--(7);\draw(4)--(7);
		\end{tikzpicture}
	\fi
}

\newcommand{\figfou}{
	\ifnum\aspdf=1
		\begin{array}{ccccc}
			\stackrel{1\oplus1}{\includegraphics[scale=0.9]{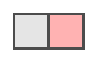}}&
			\stackrel{\kern+0.9em\dot{0}\oplus2}{\includegraphics[scale=0.9]{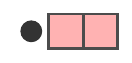}}&
			\stackrel{\kern+0.9em\dot{1}\oplus1}{\includegraphics[scale=0.9]{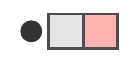}}&
			\stackrel{\kern-0.9em1\oplus\dot{1}}{\includegraphics[scale=0.9]{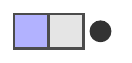}}&
			\stackrel{\kern-0.9em2\oplus\dot{0}}{\includegraphics[scale=0.9]{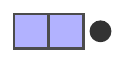}}
		\end{array}
	\else
		\tabpart{\snode\&\snode[red!30]}
		\tabpart{\dnode\&\snode[red!30]\&\snode[red!30]}
		\tabpart{\dnode\&\snode\&\snode[red!30]}
		\tabpart{\snode[blue!30]\&\snode\&\dnode}
		\tabpart{\snode[blue!30]\&\snode[blue!30]\&\dnode}
	\fi
}

\newcommand{\figfiv}{
	\ifnum\aspdf=1
		\[\begin{array}{ccc}
			\stackrel{(\dot{0},\dot{6},\dot{3},\dot{4})}{\includegraphics{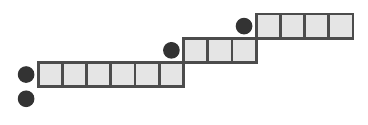}}&
			\stackrel{(\dot{0},\dot{7},\dot{2},\dot{4})}{\includegraphics{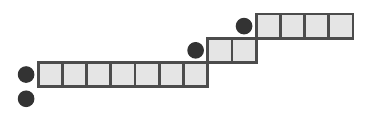}}&
			\stackrel{(\dot{1},\dot{5},\dot{3},\dot{4})}{\includegraphics{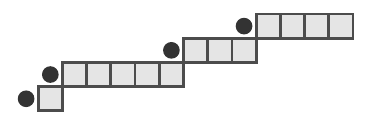}}\\
			\stackrel{(\dot{1},\dot{6},\dot{2},\dot{4})}{\includegraphics{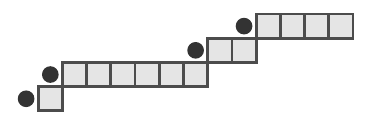}}&
			\stackrel{(\dot{3},\dot{4},\dot{2},\dot{4})}{\includegraphics{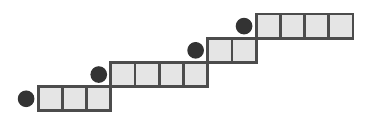}}&
			\stackrel{(\dot{3},\dot{3},\dot{3},\dot{4})}{\includegraphics{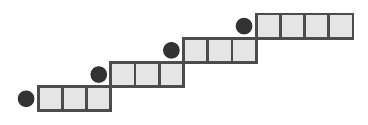}}
		\end{array}\]
	\else
		\[\begin{array}{ccc}
			\stackrel{(\dot{0},\dot{6},\dot{3},\dot{4})}{\tabpart{
				\&\&\&\&\&\&\&\&\&\dnode\&\snode\&\snode\&\snode\&\snode\\
				\&\&\&\&\&\&\dnode\&\snode\&\snode\&\snode\\
				\dnode\&\snode\&\snode\&\snode\&\snode\&\snode\&\snode\\
				\dnode
			}}&
			\stackrel{(\dot{0},\dot{7},\dot{2},\dot{4})}{\tabpart{
				\&\&\&\&\&\&\&\&\&\dnode\&\snode\&\snode\&\snode\&\snode\\
				\&\&\&\&\&\&\&\dnode\&\snode\&\snode\\
				\dnode\&\snode\&\snode\&\snode\&\snode\&\snode\&\snode\&\snode\\
				\dnode
			}}&
			\stackrel{(\dot{1},\dot{5},\dot{3},\dot{4})}{\tabpart{
				\&\&\&\&\&\&\&\&\&\dnode\&\snode\&\snode\&\snode\&\snode\\
				\&\&\&\&\&\&\dnode\&\snode\&\snode\&\snode\\
				\&\dnode\&\snode\&\snode\&\snode\&\snode\&\snode\\
				\dnode\&\snode
			}}\\
			\stackrel{(\dot{1},\dot{6},\dot{2},\dot{4})}{\tabpart{
				\&\&\&\&\&\&\&\&\&\dnode\&\snode\&\snode\&\snode\&\snode\\
				\&\&\&\&\&\&\&\dnode\&\snode\&\snode\\
				\&\dnode\&\snode\&\snode\&\snode\&\snode\&\snode\&\snode\\
				\dnode\&\snode
			}}&
			\stackrel{(\dot{3},\dot{4},\dot{2},\dot{4})}{\tabpart{
				\&\&\&\&\&\&\&\&\&\dnode\&\snode\&\snode\&\snode\&\snode\\
				\&\&\&\&\&\&\&\dnode\&\snode\&\snode\\
				\&\&\&\dnode\&\snode\&\snode\&\snode\&\snode\\
				\dnode\&\snode\&\snode\&\snode
			}}&
			\stackrel{(\dot{3},\dot{3},\dot{3},\dot{4})}{\tabpart{
				\&\&\&\&\&\&\&\&\&\dnode\&\snode\&\snode\&\snode\&\snode\\
				\&\&\&\&\&\&\dnode\&\snode\&\snode\&\snode\\
				\&\&\&\dnode\&\snode\&\snode\&\snode\\
				\dnode\&\snode\&\snode\&\snode
			}}
		\end{array}\]
	\fi

}

\newcommand{\figsix}{
	\ifnum\aspdf=1
		\[\begin{array}{cccccc}
		\quad\includegraphics{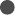}\quad&
		\quad\includegraphics{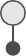}\quad&
		\quad\includegraphics{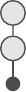}\quad&
		\quad\includegraphics{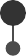}\quad&
		\quad\includegraphics{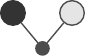}\quad&
		\quad\includegraphics{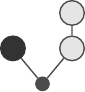}\quad\\
		1&H_1&H_2&\tilde{H}_0&\tilde{H}_1&\tilde{H}_2
		\end{array}\]
	\else
		\begin{center}
			\begin{tikzpicture}[trees]
				\node[rnode]{}
			;\end{tikzpicture}
			\qquad
			\begin{tikzpicture}[trees]
				\node[rnode]{}
				child[child]{node[tnode]{}}
			;\end{tikzpicture}
			\qquad
			\begin{tikzpicture}[trees]
				\node[rnode]{}
				child[child]{node[tnode]{}
					\tnode
				}
			;\end{tikzpicture}
			\qquad
			\begin{tikzpicture}[trees]
				\node[rnode]{}
				child[child]{node[fnode]{}}
			;\end{tikzpicture}
			\qquad
			\begin{tikzpicture}[trees]
				\node[rnode]{}
				child[child]{node[tnode]{}}
				child[child]{node[fnode]{}}
			;\end{tikzpicture}
			\qquad
			\begin{tikzpicture}[trees]
				\node[rnode]{}
				child[child]{node[tnode]{}
					\tnode
				}
				child[child]{node[fnode]{}}
			;\end{tikzpicture}
		\end{center}
	\fi
}

\newcommand{\figsev}{
	\ifnum\aspdf=1
		\[t_{\alpha}\quad=\quad\vcdraw{
		\includegraphics{pics/023.pdf}\qquad
		\includegraphics{pics/021.pdf}\qquad
		\includegraphics{pics/020.pdf}\qquad
		\includegraphics{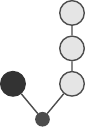}
		}\]
	\else
		\begin{tikzpicture}[trees]
			\node[rnode]{}
			child[child]{node[tnode]{}
				\tnode
				\tnode
			}
			child[child]{node[fnode]{}}
		;\end{tikzpicture}
	\fi
}

\newcommand{\figeig}{
	\ifnum\aspdf=1\[\begin{array}{rcl}
		\Delta(\gen{H}{-0.35cm}{1.2em}{\includegraphics{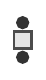}})&=&
		\gen{H}{-0.35cm}{1.2em}{\includegraphics{pics/025.pdf}}\otimes1+
		\gen{H}{-0.35cm}{1.4em}{\includegraphics{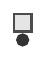}}\otimes
		\gen{H}{-0.35cm}{1.2em}{\includegraphics{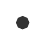}}+
		\gen{H}{-0.35cm}{1.2em}{\includegraphics{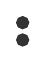}}\otimes
		\gen{H}{-0.35cm}{1.4em}{\includegraphics{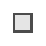}}+
		\gen{H}{-0.35cm}{1.2em}{\includegraphics{pics/029.pdf}}\otimes
		\gen{H}{-0.35cm}{1.2em}{\includegraphics{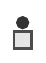}}\\&&-
		\gen{H}{-0.35cm}{1.2em}{\includegraphics{pics/027.pdf}}\otimes
		\gen{H}{-0.35cm}{1.2em}{\includegraphics{pics/029.pdf}}+
		\gen{H}{-0.35cm}{1.4em}{\includegraphics{pics/030.pdf}}\otimes
		\gen{H}{-0.35cm}{1.2em}{\includegraphics{pics/028.pdf}}-
		\gen{H}{-0.35cm}{1.2em}{\includegraphics{pics/029.pdf}}\otimes
		\gen{H}{-0.35cm}{1.4em}{\includegraphics{pics/026.pdf}}+1\otimes
		\gen{H}{-0.35cm}{1.2em}{\includegraphics{pics/025.pdf}}.
	\end{array}\\[-0.2cm]\]\else
		$\tabpart{\dnode\\\snode\\\dnode}$
		$\tabpart{\snode\\\dnode}$
		$\tabpart{\dnode\\\snode}$
		$\tabpart{\dnode\\\dnode}$
		$\tabpart{\dnode\\\snode\\\dnode}$
		$\tabpart{\dnode}$
		$\tabpart{\snode}$
	\fi
}

\newcommand{\fignin}{
	\ifnum\aspdf=1\[
		S(\kern-1em\gen{H}{-0.35cm}{2.3em}{\includegraphics{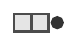}})=\kern+0.5em-\kern-0.7em
		\gen{H}{-0.35cm}{2.3em}{\includegraphics{pics/031.pdf}}+\kern-0.5em
		\gen{H}{-0.35cm}{1.8em}{\includegraphics{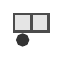}}+
		\gen{H}{-0.35cm}{1.0em}{\includegraphics{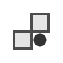}}+\kern-0.5em
		\gen{H}{-0.35cm}{1.8em}{\includegraphics{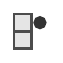}}+
		\gen{H}{-0.35cm}{1.0em}{\includegraphics{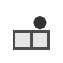}}-
		\gen{H}{-0.35cm}{1.4em}{\includegraphics{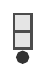}}-
		\gen{H}{-0.35cm}{1.4em}{\includegraphics{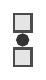}}-
		\gen{H}{-0.35cm}{1.2em}{\includegraphics{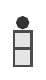}}.
	\]\else
		$\tabpart{\snode\&\snode\&\dnode}$
		$\tabpart{\snode\&\snode\\\dnode}$
		$\tabpart{\&\snode\\\snode\&\dnode}$
		$\tabpart{\snode\&\dnode\\\snode}$
		$\tabpart{\&\dnode\\\snode\&\snode}$
		$\tabpart{\snode\\\snode\\\dnode}$
		$\tabpart{\snode\\\dnode\\\snode}$
		$\tabpart{\dnode\\\snode\\\snode}$
	\fi
}

\newcommand{\figten}{
	\ifnum\aspdf=1\[
		\vcdraw{\includegraphics[scale=0.9]{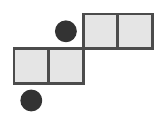}}\odot
		\vcdraw{\includegraphics[scale=0.9]{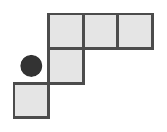}}=
		\vcdraw{\includegraphics[scale=0.9]{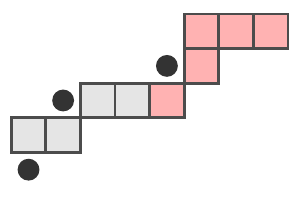}}
	\]\else
		$\tabpart{\&\dnode\&\snode\&\snode\\\snode\&\snode\\\dnode}$
		$\tabpart{\&\snode\&\snode\&\snode\\\dnode\&\snode\\\snode}$
		$\tabpart{\&\&\&\&\&\&\snode[red!30]\&\snode[red!30]\&\snode[red!30]\\\&\&\&\&\&\dnode\&\snode[red!30]\\\&\&\dnode\&\snode\&\snode\&\snode[red!30]\\\&\snode\&\snode\\\&\dnode}$
	\fi
}

\newcommand{\figele}{
	\ifnum\aspdf=1\[
		\Psi_3\,\,=\kern-1.1em
		\gen{H}{-0.35cm}{2.5em}{\includegraphics{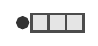}}-
		\gen{H}{-0.35cm}{0.2em}{\includegraphics{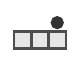}}+
		\gen{H}{-0.35cm}{0.7em}{\includegraphics{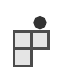}}+
		\gen{H}{-0.35cm}{0.7em}{\includegraphics{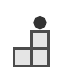}}-
		\gen{H}{-0.35cm}{1.2em}{\includegraphics{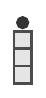}}-\kern-0.4em
		\gen{H}{-0.35cm}{1.2em}{\includegraphics{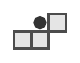}}+\kern-0.1em
		\gen{H}{-0.35cm}{1.6em}{\includegraphics{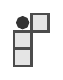}}-\kern-0.8em
		\gen{H}{-0.35cm}{2.0em}{\includegraphics{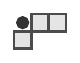}}.
	\]\else
		\tabpart{\dnode\&\snode\&\snode\&\snode}
		\tabpart{\&\&\dnode\\\snode\&\snode\&\snode}
		\tabpart{\&\dnode\\\snode\&\snode\\\snode}
		\tabpart{\&\dnode\\\&\snode\\\snode\&\snode}
		\tabpart{\dnode\\\snode\\\snode\\\snode}
		\tabpart{\&\dnode\&\snode\\\snode\&\snode}
		\tabpart{\dnode\&\snode\\\snode\\\snode}
		\tabpart{\dnode\&\snode\&\snode\\\snode}
	\fi
}

\newcommand{\figtwe}{
	\ifnum\aspdf=1
		\[\includegraphics{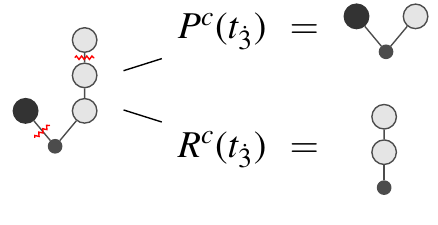}\]
	\else
		\begin{tikzpicture}
			\node at(0,0)(0){
				\begin{tikzpicture}[trees]
					\node[rnode]{}
					child[child]{node[tnode]{}
						\tnode
						\tnode
					}
					child[child]{node[fnode]{}};
					\draw[cuts](0.2,0.9)to(0.4,0.9);
					\draw[cuts,rotate=41](-0.1,0.2)to(0.1,0.2);
				\end{tikzpicture}\,\,
			};
			\node[align=center,inner sep=0.0pt]at(3.25,+0.6)(1){
				\,\,\,\begin{tikzpicture}[trees]
					\node[rnode]{}
					child[child]{node[tnode]{}}
					child[child]{node[fnode]{}};
				\end{tikzpicture}
			};			
			\node[align=center,inner sep=8.5pt]at(3.25,-0.6)(2){
				\,\,\,\kern-0.1em\begin{tikzpicture}[trees]
					\node[rnode]{}
					child[child]{node[tnode]{}
						\tnode
					};
				\end{tikzpicture}
			};
			\node at(1.9,+0.6)(4){$P^c(t_{\dot{3}})\,\,=$};
			\node at(1.9,-0.6)(5){$R^c(t_{\dot{3}})\,\,=$};
			\draw(0)--(4);\draw(0)--(5);
		\end{tikzpicture}
	\fi
}

\newcommand{\figthi}{
	\ifnum\aspdf=1\[
		t_{\alpha}\,=
		\includegraphics{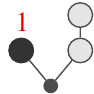}\kern+0.8em
		\includegraphics{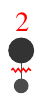}\kern+0.8em
		\includegraphics{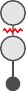}\kern+0.8em
		\includegraphics{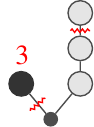}\kern+0.8em
		\includegraphics{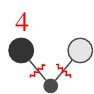}\kern+0.8em
		\qquad
		P^c(t_{\alpha})\,=
		\includegraphics{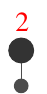}\kern+0.8em
		\includegraphics{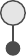}\kern+0.8em
		\includegraphics{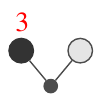}\kern+0.8em
		\includegraphics{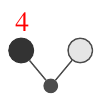}\kern+0.8em
		\qquad
		R^c(t_{\alpha})\,=
		\includegraphics{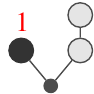}\kern+0.8em
		\includegraphics{pics/019.pdf}\kern+0.8em
		\includegraphics{pics/020.pdf}
	\]\else
		\begin{tikzpicture}[trees]
			\node[rnode]{}
			child[child]{node[tnode]{}
				\tnode
			}
			child[child]{node[fnode]{}};
			\node at(-0.28,0.65){$_{\red1}$};
		\end{tikzpicture}\,\,\,
		\begin{tikzpicture}[trees]
			\node[rnode]{}
			child[child]{node[fnode]{}};
			\node at(0.02,0.65){$_{\red2}$};
			\draw[cuts](-0.1,0.16)to(0.1,0.16);
		\end{tikzpicture}\,\,\,
		\begin{tikzpicture}[trees]
			\node[rnode]{}
			child[child]{node[tnode]{}
				\tnode
			};
			\draw[cuts](-0.1,0.54)to(0.1,0.54);
		\end{tikzpicture}\,\,\,
		\begin{tikzpicture}[trees]
			\node[rnode]{}
			child[child]{node[tnode]{}
				\tnode
				\tnode
			}
			child[child]{node[fnode]{}};
			\node at(-0.28,0.65){$_{\red3}$};
			\draw[cuts](0.2,0.9)to(0.4,0.9);
			\draw[cuts,rotate=41](-0.1,0.2)to(0.1,0.2);
		\end{tikzpicture}\,\,\,
		\begin{tikzpicture}[trees]
			\node[rnode]{}
			child[child]{node[tnode]{}}
			child[child]{node[fnode]{}};
			\node at(-0.28,0.65){$_{\red4}$};
			\draw[cuts,rotate=41](-0.1,0.2)to(0.1,0.2);
			\draw[cuts,rotate=-41](-0.1,0.2)to(0.1,0.2);
		\end{tikzpicture}\,\,\,
		\quad
		\begin{tikzpicture}[trees]
			\node[rnode]{}
			child[child]{node[fnode]{}};
			\node at(0.02,0.65){$_{\red2}$};
		\end{tikzpicture}\,\,\,
		\begin{tikzpicture}[trees]
			\node[rnode]{}
			child[child]{node[tnode]{}};
		\end{tikzpicture}\,\,\,
		\begin{tikzpicture}[trees]
			\node[rnode]{}
			child[child]{node[tnode]{}}
			child[child]{node[fnode]{}};
			\node at(-0.28,0.65){$_{\red3}$};
		\end{tikzpicture}\,\,\,
		\begin{tikzpicture}[trees]
			\node[rnode]{}
			child[child]{node[tnode]{}}
			child[child]{node[fnode]{}};
			\node at(-0.28,0.65){$_{\red4}$};
		\end{tikzpicture}\,\,\,
		\quad
		\begin{tikzpicture}[trees]
			\node[rnode]{}
			child[child]{node[tnode]{}
				\tnode
			}
			child[child]{node[fnode]{}};
			\node at(-0.28,0.65){$_{\red1}$};
		\end{tikzpicture}\,\,\,
		\begin{tikzpicture}[trees]
			\node[rnode]{}
			child[child]{node[tnode]{}};
		\end{tikzpicture}\,\,\,
		\begin{tikzpicture}[trees]
			\node[rnode]{}
			child[child]{node[tnode]{}
				\tnode
			};
		\end{tikzpicture}
	\fi
}

\newcommand{\figftn}{
	\ifnum\aspdf=1\[\begin{array}{cccccc}
		\includegraphics[scale=0.9]{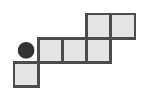}&
		\includegraphics[scale=0.9]{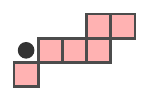}&
		\includegraphics[scale=0.9]{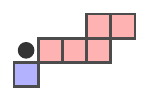}&
		\includegraphics[scale=0.9]{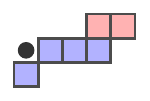}&
		\includegraphics[scale=0.9]{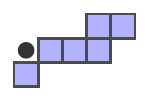}&
		\includegraphics[scale=0.9]{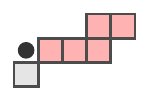}\\[-1.2mm]
		^{(1,\dot{3},2)}&^{()(1,\dot{3},2)}&^{(1)(\dot{3},2)}&^{(1,\dot{3})(2)}&^{(1,\dot{3},2)()}&^{(1,\dot{0})\odot(3,2)}\\
		\includegraphics[scale=0.9]{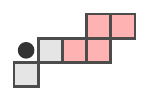}&
		\includegraphics[scale=0.9]{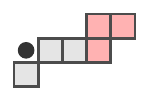}&
		\includegraphics[scale=0.9]{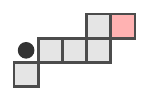}&
		\includegraphics[scale=0.9]{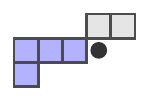}&
		\includegraphics[scale=0.9]{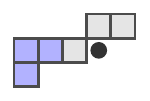}&
		\includegraphics[scale=0.9]{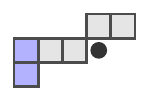}\\[-1.2mm]
		^{(1,\dot{1})\odot(2,2)}&^{(1,\dot{2})\odot(1,2)}&^{(1,\dot{3},1)\odot(1)}&^{(1,3)\odot(\dot{0},2)}&^{(1,2)\odot(\dot{1},2)}&^{(1,1)\odot(\dot{2},2)}
	\end{array}\]\else
		\tabpart{
			\&\&\&\snode\&\snode\\
			\dnode\&\snode\&\snode\&\snode\\
			\snode
		}
		\tabpart{
			\&\&\&\snode[red!30]\&\snode[red!30]\\
			\dnode\&\snode[red!30]\&\snode[red!30]\&\snode[red!30]\\
			\snode[red!30]
		}
		\tabpart{
			\&\&\&\snode[red!30]\&\snode[red!30]\\
			\dnode\&\snode[red!30]\&\snode[red!30]\&\snode[red!30]\\
			\snode[blue!30]
		}
		\tabpart{
			\&\&\&\snode[red!30]\&\snode[red!30]\\
			\dnode\&\snode[blue!30]\&\snode[blue!30]\&\snode[blue!30]\\
			\snode[blue!30]
		}
		\tabpart{
			\&\&\&\snode[blue!30]\&\snode[blue!30]\\
			\dnode\&\snode[blue!30]\&\snode[blue!30]\&\snode[blue!30]\\
			\snode[blue!30]
		}
		\tabpart{
			\&\&\&\snode[red!30]\&\snode[red!30]\\
			\dnode\&\snode[red!30]\&\snode[red!30]\&\snode[red!30]\\
			\snode
		}
		\tabpart{
			\&\&\&\snode[red!30]\&\snode[red!30]\\
			\dnode\&\snode\&\snode[red!30]\&\snode[red!30]\\
			\snode
		}
		\tabpart{
			\&\&\&\snode[red!30]\&\snode[red!30]\\
			\dnode\&\snode\&\snode\&\snode[red!30]\\
			\snode
		}
		\\[1cm]
		\tabpart{
			\&\&\&\snode\&\snode[red!30]\\
			\dnode\&\snode\&\snode\&\snode\\
			\snode
		}
		\tabpart{
			\&\&\&\snode\&\snode\\
			\snode[blue!30]\&\snode[blue!30]\&\snode[blue!30]\&\dnode\\
			\snode[blue!30]
		}
		\tabpart{
			\&\&\&\snode\&\snode\\
			\snode[blue!30]\&\snode[blue!30]\&\snode\&\dnode\\
			\snode[blue!30]
		}
		\tabpart{
			\&\&\&\snode\&\snode\\
			\snode[blue!30]\&\snode\&\snode\&\dnode\\
			\snode[blue!30]
		}
	\fi
}

\newcommand{\figfif}{
	\ifnum\aspdf=1\[\begin{array}{rcl}
		\Delta(\kern-0.3em\gen{L}{-0.35cm}{0.4em}{\includegraphics{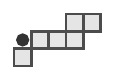}})&\!\!\!=\!\!\!&1\otimes\kern-0.7em
		\gen{L}{-0.35cm}{0.4em}{\includegraphics{pics/073.pdf}}+\!
		\gen{L}{-0.35cm}{1.3em}{\includegraphics{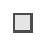}}\kern-0.1em\otimes\kern-0.7em
		\gen{L}{-0.35cm}{0.4em}{\includegraphics{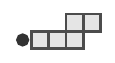}}+\kern-1.4em
		\gen{L}{-0.35cm}{2.4em}{\includegraphics{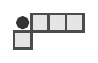}}\kern-0.1em\otimes\kern-0.9em
		\gen{L}{-0.35cm}{2.0em}{\includegraphics{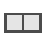}}+\!
		\gen{L}{-0.35cm}{0.4em}{\includegraphics{pics/073.pdf}}\otimes1\,+
		\gen{L}{-0.35cm}{1.0em}{\includegraphics{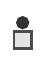}}\kern-0.3em\otimes\kern-0.7em
		\gen{L}{-0.35cm}{1.0em}{\includegraphics{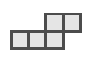}}+\kern-0.4em
		\gen{L}{-0.35cm}{1.5em}{\includegraphics{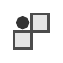}}\kern-0.1em\otimes\kern-0.7em
		\gen{L}{-0.35cm}{1.4em}{\includegraphics{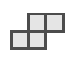}}\\&&+\kern-0.8em
		\gen{L}{-0.35cm}{1.9em}{\includegraphics{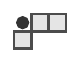}}\!\otimes\kern-0.8em
		\gen{L}{-0.35cm}{1.8em}{\includegraphics{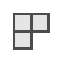}}\,\,+\,\,
		\gen{L}{-0.35cm}{-0.05em}{\includegraphics{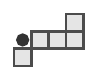}}\!\otimes\kern-0.2em
		\gen{L}{-0.35cm}{1.35em}{\includegraphics{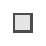}}\,+\kern-1.1em
		\gen{L}{-0.35cm}{2.3em}{\includegraphics{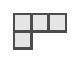}}\!\otimes\kern-0.7em
		\gen{L}{-0.35cm}{1.8em}{\includegraphics{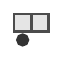}}+\kern-0.5em
		\gen{L}{-0.35cm}{1.8em}{\includegraphics{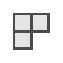}}\otimes\kern-0.7em
		\gen{L}{-0.35cm}{1.3em}{\includegraphics{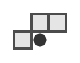}}+
		\gen{L}{-0.35cm}{1.3em}{\includegraphics{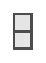}}\otimes\kern-0.6em
		\gen{L}{-0.35cm}{0.9em}{\includegraphics{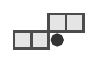}}.
	\end{array}\]\else
		\tabpart{
			\snode
		}
		\tabpart{
			\&\&\&\snode\&\snode\\
			\dnode\&\snode\&\snode\&\snode
		}
		\tabpart{
			\dnode\&\snode\&\snode\&\snode\\
			\snode
		}
		\tabpart{
			\&\&\&\snode\&\snode
		}
		\tabpart{
			\dnode\\
			\snode
		}
		\tabpart{
			\&\&\&\snode\&\snode\\
			\&\snode\&\snode\&\snode
		}
		\tabpart{
			\dnode\&\snode\\
			\snode
		}
		\tabpart{
			\&\&\snode\&\snode\\
			\&\snode\&\snode
		}
		\tabpart{
			\dnode\&\snode\&\snode\\
			\snode
		}
		\tabpart{
			\snode\&\snode\\
			\snode\\
		}
		\tabpart{
			\&\&\&\snode\\
			\dnode\&\snode\&\snode\&\snode\\
			\snode
		}
		\tabpart{
			\snode
		}
		\tabpart{
			\snode\&\snode\&\snode\\
			\snode
		}
		\tabpart{
			\snode\&\snode\\
			\dnode
		}
		\tabpart{
			\snode\&\snode\\
			\snode
		}
		\tabpart{
			\&\snode\&\snode\\
			\snode\&\dnode
		}
		\tabpart{
			\snode\\
			\snode
		}
		\tabpart{
			\&\&\snode\&\snode\\
			\snode\&\snode\&\dnode
		}
	\fi
}

\newcommand{\figstn}{
	\ifnum\aspdf=1\[
		\alpha\,\,=\!\!\vcdraw{\includegraphics{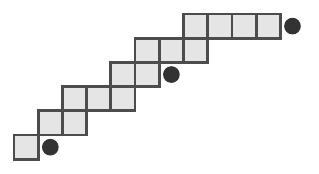}}\qquad
		\tilde{\alpha}\,\,=\,\vcdraw{\includegraphics{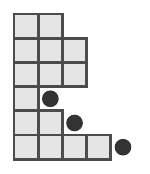}}
	\]\else
		\tabpart{
			\&\&\&\&\&\&\&\snode\&\snode\&\snode\&\snode\&\dnode\\
			\&\&\&\&\&\snode\&\snode\&\snode\\
			\&\&\&\&\snode\&\snode\&\dnode\\
			\&\&\snode\&\snode\&\snode\\
			\&\snode\&\snode\\\
			\snode\&\dnode
		}
		\tabpart{
			\snode\&\snode\\\
			\snode\&\snode\&\snode\\
			\snode\&\snode\&\snode\\
			\snode\&\dnode\\
			\snode\&\snode\&\dnode\\			
			\snode\&\snode\&\snode\&\snode\&\dnode
		}
	\fi
}

\newcommand{\figsvt}{
	\ifnum\aspdf=1\[\begin{array}{rcl}
		\Delta(\,
			\includegraphics{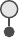}\,\,
			\includegraphics{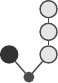}\,\,
			\includegraphics{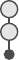}
		\,)&\!=\!&
		\includegraphics{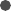}\otimes
		\includegraphics{pics/094-1U.pdf}\,\,
		\includegraphics{pics/094-3D.pdf}\,\,
		\includegraphics{pics/094-2U.pdf}\,\,+\,\,
		\includegraphics{pics/094-1U.pdf}\otimes
		\includegraphics{pics/094-3D.pdf}\,\,
		\includegraphics{pics/094-2U.pdf}\,\,+\,\,
		\includegraphics{pics/094-1U.pdf}\,\,
		\includegraphics{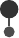}\otimes
		\includegraphics{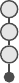}\,\,
		\includegraphics{pics/094-2U.pdf}\,\,+\,\,
		\includegraphics{pics/094-1U.pdf}\,\,
		\includegraphics{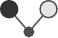}\otimes
		\includegraphics{pics/094-2U.pdf}\,\,
		\includegraphics{pics/094-2U.pdf}\,\,+\,\,
		\includegraphics{pics/094-1U.pdf}\,\,
		\includegraphics{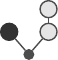}\otimes
		\includegraphics{pics/094-1U.pdf}\,\,
		\includegraphics{pics/094-2U.pdf}\,\,+\,\,
		\includegraphics{pics/094-1U.pdf}\,\,
		\includegraphics{pics/094-3D.pdf}\otimes
		\includegraphics{pics/094-2U.pdf}\kern+2cm\\[0.2cm]&&+\,\,
		\includegraphics{pics/094-1U.pdf}\,\,
		\includegraphics{pics/094-1U.pdf}\otimes
		\includegraphics{pics/094-2D.pdf}\,\,
		\includegraphics{pics/094-2U.pdf}\,\,+\,\,
		\includegraphics{pics/094-1U.pdf}\,\,
		\includegraphics{pics/094-2U.pdf}\otimes
		\includegraphics{pics/094-1D.pdf}\,\,
		\includegraphics{pics/094-2U.pdf}\,\,+\,\,
		\includegraphics{pics/094-1U.pdf}\,\,
		\includegraphics{pics/094-3U.pdf}\otimes
		\includegraphics{pics/094-0D.pdf}\,\,
		\includegraphics{pics/094-2U.pdf}\,\,+\,\,
		\includegraphics{pics/094-1U.pdf}\,\,
		\includegraphics{pics/094-3D.pdf}\,\,
		\includegraphics{pics/094-1U.pdf}\otimes
		\includegraphics{pics/094-1U.pdf}\,\,+\,\,
		\includegraphics{pics/094-1U.pdf}\,\,
		\includegraphics{pics/094-3D.pdf}\,\,
		\includegraphics{pics/094-2U.pdf}\otimes
		\includegraphics{pics/094-0U.pdf}
	\end{array}\]\else
		\begin{tikzpicture}[trees]
			\node[rnode]{}
		;\end{tikzpicture}
		\begin{tikzpicture}[trees]
			\node[rnode]{}
			child[child]{node[tnode]{}}
		;\end{tikzpicture}
		\begin{tikzpicture}[trees]
			\node[rnode]{}
			child[child]{node[tnode]{}
				\tnode
			}
		;\end{tikzpicture}
		\begin{tikzpicture}[trees]
			\node[rnode]{}
			child[child]{node[tnode]{}
				\tnode
				\tnode
			}
		;\end{tikzpicture}
		\begin{tikzpicture}[trees]
			\node[rnode]{}
			child[child]{node[fnode]{}};
		;\end{tikzpicture}
		\begin{tikzpicture}[trees]
			\node[rnode]{}
			child[child]{node[tnode]{}}
			child[child]{node[fnode]{}};
		;\end{tikzpicture}
		\begin{tikzpicture}[trees]
			\node[rnode]{}
			child[child]{node[tnode]{}
				\tnode
			}
			child[child]{node[fnode]{}};
		;\end{tikzpicture}
		\begin{tikzpicture}[trees]
			\node[rnode]{}
			child[child]{node[tnode]{}
				\tnode
				\tnode
			}
			child[child]{node[fnode]{}};
		;\end{tikzpicture}
	\fi
}